\documentclass[final,leqno,showlabe]{siamltex}
\usepackage{amsmath,amssymb}
\usepackage{epsfig,color}

\usepackage{hyperref}
\usepackage[utf8]{inputenc}

\usepackage[english]{babel}
\usepackage{mathrsfs}
\usepackage{amsfonts}
\usepackage{color}
\usepackage{subfigure}
\usepackage{graphicx}

\numberwithin{equation}{section}

\newcommand{\gm}{\gamma}

\newcommand{\FF}{\mathcal{F}}
\newcommand{\be}{\begin{equation}}
\newcommand{\ee}{\end{equation}}
\newcommand{\bea}{\begin{eqnarray}}
\newcommand{\eea}{\end{eqnarray}}

\newcommand{\lmd}{\lambda}
\newcommand{\tu}{\tilde{u}}
\newcommand{\tbx}{\tilde{\bx}}
\newcommand{\tby}{\tilde{\by}}
\newcommand{\tphi}{\tilde{\phi}}

\newcommand{\tOmega}{\tilde{\Omega}}
\newcommand{\tdelta}{\tilde{\delta}}

\def\bx{\mathbf{x}}
\def\bk{\mathbf{k}}
\def\bm{\mathbf{m}}
\def\by{\mathbf{y}}
\def\bog{{_\Omega}}
\def\botg{{_{\tilde \Omega}}}

\newtheorem{remark}{Remark}[section]

\title{Fundamental Gaps of the Fractional Schr\"{o}dinger Operator}
\author{Weizhu Bao\thanks{Department of Mathematics, National University of Singapore, Singapore 119076, Singapore (matbaowz@nus.edu.sg, URL: http://www.math.nus.edu.sg/\~{}bao/). This author's research was supported by the Academic Research Fund of Ministry of Education of Singapore grant
R-146-000-223-112.}
\and Xinran Ruan\thanks{Corresponding Author.
Department of Mathematics, National University of Singapore, Singapore 119076, Singapore (ruan@ljll.math.upmc.fr). 
Current address: Laboratoire J.-L. Lions, Universit\'{e} Pierre et 
Marie Curie, 75252 Paris cedex 05, France. 
}
\and Jie Shen\thanks{Department of Mathematics, Purdue University, West Lafayette, IN 47907-1957, USA (shen7@purdue.edu).
}
\and Changtao Sheng\thanks{Fujian Provincial Key Laboratory on Mathematical Modeling {\rm \&} High Performance Scientific Computing and School of Mathematical Sciences, Xiamen University, Xiamen, Fujian 361005, P.R. China (ctsheng8@gmail.com).}
}

\date{}

\begin{document}

\maketitle

\begin{abstract}
We study asymptotically and numerically the fundamental gap -- the difference between the first two smallest (and distinct) eigenvalues -- of the fractional Schr\"{o}dinger operator (FSO) and formulate a gap conjecture on the fundamental gap of the FSO. We begin with an introduction
of the FSO on bounded domains with homogeneous Dirichlet boundary conditions, while the fractional Laplacian operator defined either via the local fractional Laplacian (i.e. via the eigenfunctions decomposition of the Laplacian operator) or via the classical fractional Laplacian (i.e. zero extension of the eigenfunctions outside the bounded domains and then via the Fourier transform). For the FSO on bounded domains with either the local fractional Laplacian or the classical fractional Laplacian, we obtain the fundamental gap of the FSO analytically on simple geometry without potential and numerically on complicated geometries and/or with different convex potentials.
Based on the asymptotic and extensive numerical results, a gap conjecture on the fundamental gap of the FSO is formulated.
Surprisingly, for two and higher dimensions, the lower bound of the fundamental gap depends not only on the diameter of the domain, but also the diameter of the largest inscribed ball of the domain, which is completely different from the case of the Schr\"{o}dinger operator.
Extensions of these results for the FSO in the whole space and on bounded domains with periodic boundary conditions are presented.
\end{abstract}



\pagestyle{myheadings} \markboth{W. Bao, X. Ruan, J. Shen and C. Sheng}
{Fundamental gaps of fractional Schr\"{o}dinger operators}

\maketitle

\section{Introduction}  \setcounter{equation}{0}

Consider the {\bf fractional Schr\"{o}dinger operator} (FSO) in
$n$-dimensions ($n=1,2,3$)
\begin{equation}\label{eq:fso}
L_{\rm{FSO}}\;u(\bx):=\left[(-\Delta)^{\frac{\alpha}{2}}
+V(\bx)\right]u(\bx),\qquad \bx\in \mathbb{R}^n,
\end{equation}
where $\alpha\in(0,2]$, $V(\bx)$ 
is a given real-valued function  and the fractional Laplacian operator $(-\Delta)^{\frac{\alpha}{2}}$ is defined via the Fourier transform (see \cite{Caffarelli1,Nezza} and references therein) as
\be\label{def:FL_F}
(-\Delta)^{\frac{\alpha}{2}}u(\bx)=\mathcal{F}^{-1}
(|\bk|^{\alpha}(\FF u)(\bk)), \qquad \bx,\bk\in \mathbb{R}^n,
\ee
with $\FF$ and $\mathcal{F}^{-1}$ the Fourier transform and inverse Fourier transform, respectively.
Obviously when $\alpha=2$, \eqref{eq:fso} becomes the (classical) Schr\"{o}dinger operator. When $n=2$ and $\alpha=1$, it is related to
the square-root Laplacian operator which is used for the Coulumb interaction and dipole-dipole interaction in two dimensions (2D) \cite{BJMZ,BBC,CRLB}. In fact, the Schr\"{o}dinger equation governed by the Schr\"{o}dinger operator can be interpreted via the Feynman path integral approach over Brownian-like quantum paths \cite{Feynman1, Feynman2}.
When the method is generalized to be over the L\'{e}vy-like quantum mechanic path,  Nick Laskin derived the fractional Schr\"{o}dinger equation, where the Schr\"{o}dinger operator is replaced by the fractional one \cite{Laskin1, Laskin2,Laskin3}, i.e. $L_{\rm FSO}$.  And the new model derived lays the foundation of the fractional quantum mechanics.

It can be shown that, with definition \eqref{def:FL_F}, $\lim_{\alpha\to2^-}(-\Delta)^{\alpha/2}u=-\Delta u$ and $\lim_{\alpha\to0^+}(-\Delta)^{\alpha/2}u=u$ \cite{Nezza,Podlubny,Samko,Stinga}.
The above definition is easy to understand and useful for problems defined in the whole space. However, it is hard to get local estimates from \eqref{def:FL_F}.
  An alternative way to define $(-\Delta)^{\frac{\alpha}{2}}$ is through the principle value integral (see \cite{Caffarelli1,Caffarelli2,Caffarelli3, Servadei} and references therein) as
\be\label{def:FL_I}
(-\Delta)^{\frac{\alpha}{2}}u(\bx)=C_{n,\alpha}
\int_{\mathbb{R}^n}\frac{u(\bx)-u(\by)}{|\bx-\by|^{n+\alpha}}\,d\by,
\qquad \bx\in \mathbb{R}^n,
\ee
where $C_{n,\alpha}$ is a constant whose value  can be computed explicitly as
\be
C_{n,\alpha}=\frac{2^\alpha\Gamma(n/2+\alpha/2)}{\pi^{n/2}|\Gamma(-\alpha/2)|}=\frac{\alpha\Gamma(n/2+\alpha/2)}{2^{1-\alpha}\pi^{n/2}\Gamma(1-\alpha/2)}.
\ee
It is easy to verify that $C_{n,\alpha}\approx\frac{\alpha\Gamma(n/2)}{2\pi^{n/2}}$ as $\alpha\to0^+$ and $C_{n,\alpha}\approx\frac{n\Gamma(n/2)}{\pi^{n/2}}(2-\alpha)$ as $\alpha\to2^-$.
The definition \eqref{def:FL_I} is most useful to study local properties and it is equivalent to the  definition \eqref{def:FL_F} if $u(\bx)$ is smooth enough \cite{Caffarelli1,Nezza}.

In this paper, we are interested in the eigenvalues of the FSO, i.e.
find $E\in {\mathbb R}$ and a complex-valued function $\phi:=\phi(\bx)$ such that
\begin{equation}\label{eq:eig}
L_{\rm{FSO}}\;\phi(\bx)=
\left[(-\Delta)^{\frac{\alpha}{2}}+V(\bx)\right]\phi(\bx)=E\,\phi(\bx),\qquad \bx\in \mathbb{R}^n,
\end{equation}
especially the difference between the first two smallest eigenvalues -- the fundamental gap. For simplicity of notations and without loss of generality,
we assume that $V(\bx)$ is non-negative and is taken such that the first two
smallest eigenvalues of \eqref{eq:eig} are distinct, i.e.
the eigenvalues of \eqref{eq:eig} satisfy
$0<E_1:=E_1(\alpha)<E_2:=E_2(\alpha)<\cdots$
Assume that $\phi_1^{(\alpha)}$ and $\phi_2^{(\alpha)}$ are the corresponding eigenfunctions of $E_1$ and $E_2$, respectively, then  the first two smallest eigenvalues can be computed via
the Rayleigh quotients as
\begin{equation}
E_1(\alpha)=\min_{u\ne 0} \frac{E^{(\alpha)}(u)}{\|u\|^2},
\qquad E_2(\alpha)=\min_{u\ne 0, (u,\phi_1^{(\alpha)})=0} \frac{E^{(\alpha)}(u)}{\|u\|^2},
\end{equation}
where
\be\label{def:m_E}
\begin{split}
&\|u\|^2:=\int_{\mathbb{R}^n}\,|u(\bx)|^2\,d\bx, \quad (u,v):=\int_{\mathbb{R}^n} u(\bx)^* v(\bx)\,d\bx,\\
&E^{(\alpha)}(u):=\int_{\mathbb{R}^n}\,
\left[u(\bx)^*(-\Delta)^{\alpha/2}u(\bx)+V(\bx)|u(\bx)|^2\right]\,d\bx\\
&\qquad \ \ \quad =\int_{\mathbb{R}^n}\,|\bk|^\alpha \,|(\FF u)(\bk)|^2\,d\bk+\int_{\mathbb{R}^n}\,
V(\bx)|u(\bx)|^2\,d\bx,
\end{split}
\ee
with $f^*$ denoting the complex conjugate of $f$. Since we are mainly interested in the first two eigenvalues and their difference,
without loss generality and for simplicity of notations,
we will take $\phi_1^{(\alpha)}$ and $\phi_2^{(\alpha)}$ as real-valued functions satisfying that $\phi_1^{(\alpha)}$ is non-negative and both are
normalized to $1$, i.e.
$\|\phi_1^{(\alpha)}\|=\|\phi_2^{(\alpha)}\|=1$. Then the first two eigenvalues can also be computed as
\begin{equation}
E_1(\alpha)=E^{(\alpha)}(\phi_1^{(\alpha)})
\qquad E_2(\alpha)= E^{(\alpha)}(\phi_2^{(\alpha)}).
\end{equation}
The {\bf fundamental gap} of the FSO \eqref{eq:fso} is defined as
\be\label{gapfso}
\delta(\alpha):=E_2(\alpha)-E_1(\alpha)=
E^{(\alpha)}(\phi_2^{(\alpha)})-E^{(\alpha)}(\phi_1^{(\alpha)})>0,
\qquad 0<\alpha\le 2.
\ee

Let $\Omega\subset {\mathbb{R}^n}$ be a  bounded and open domain.
When $\alpha=2$ and $V_\bog(\bx):=V(\bx)|_\Omega \in L^2(\Omega)$ and $V(\bx)=+\infty$ for $\bx\in \Omega^c:={\mathbb{R}^n}\backslash \Omega$,
the time-independent
Schr\"{o}dinger equation \eqref{eq:eig} is reduced to 
\begin{equation}\label{eq:eigb}
\begin{split}
&\left[-\Delta+V_{_{\Omega}}(\bx)\right]\phi(\bx)=\lmd\,\phi(\bx),\qquad \bx\in \Omega,\\
&\phi(\bx)=0, \qquad \bx\in \Gamma:=\partial\Omega.
\end{split}
\end{equation}
When $V_\bog(\bx)\ge0$ for $\bx\in \Omega$, all eigenvalues of the eigenvalue
problem \eqref{eq:eigb} are distinct and positive and their corresponding
eigenfunctions are orthogonal and they form a complete basis of
$L^2(\Omega)$. In this case, based on analytical results for simple geometry
and numerical results, a {\bf Gap Conjecture} on the fundamental gap of \eqref{eq:eigb} was formulated as \cite{Ashbaugh1,Ashbaugh2,Singer}: For any convex domain $\Omega$ and convex potential $V_\bog(\bx)$, one has
\begin{equation}
\delta:=\delta(2)\ge \frac{3\pi^2}{D^2},
\end{equation}
where $D$ and $d$ are the diameter of $\Omega$ and the diameter of the largest inscribed ball of $\Omega$, respectively, defined as (see Fig. \ref{diam})
\begin{equation}
D:=\max_{\bx,\by\in \overline{\Omega}}\; |\bx-\by|,\quad
d:=\sup_{\bx\in\Omega} \sup \; \Bigl\{ r>0\ | \  B_r(\bx):=\{\by\ | \ |\bx-\by|<r\} \subset \Omega\Bigr\}.
\end{equation}
\begin{figure}[htbp]
\centerline{\psfig{figure=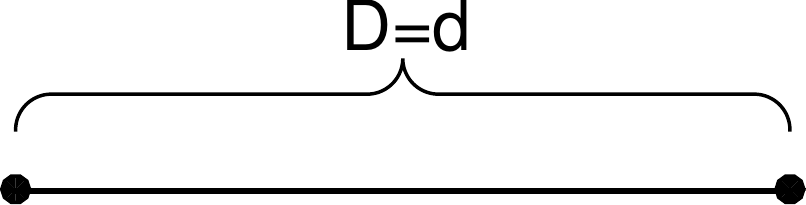,height=1cm,width=4cm,angle=0}\qquad\qquad
\psfig{figure=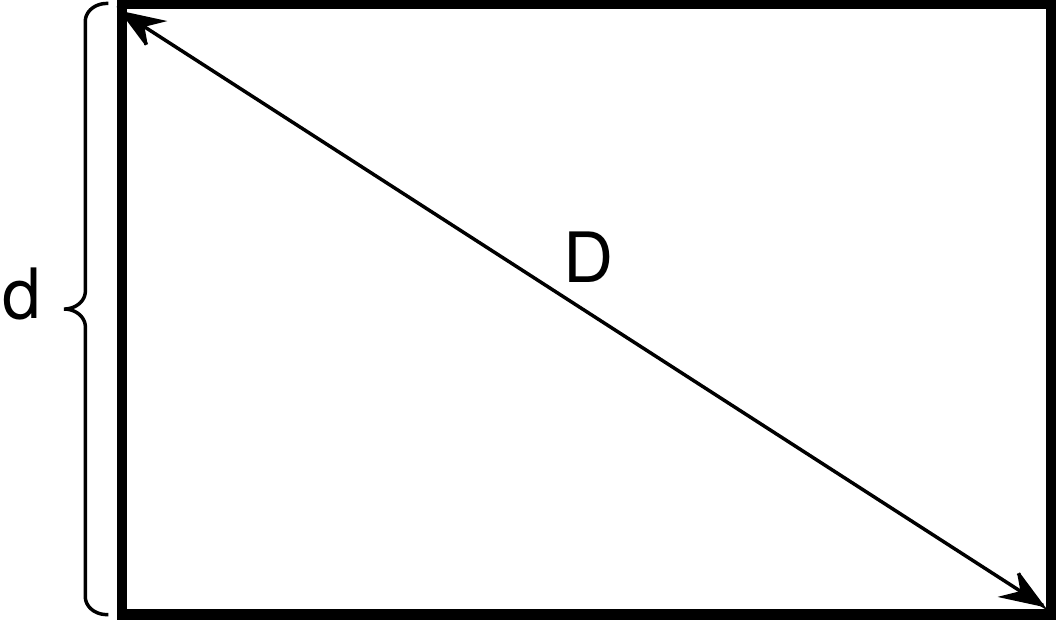,height=2cm,width=4cm,angle=0}}
\caption{Illustration of diameters $d$ and $D$ in 1D (left) and
2D (right).}
\label{diam}
\end{figure}

This gap conjecture was rigorously proved by Andrews and Clutterbuck  \cite{Andrews}. The lower bound of the fundamental gaps depends only on the diameter  of the domains and is independent of the
external potential $V(\bx)$ and the different shapes of $\Omega$.
It is noted that the gap conjecture links the algebraic property
(i.e. difference of the first two eigenvalues of the Schr\"{o}dinger operator) with the geometric property of the bounded domain $\Omega$
(i.e. its diameter). Extension of the gap conjecture
to the Schr\"{o}dinger operator in the whole space
with a harmonic-type potential, i.e. \eqref{eq:fso} with $\alpha=2$,
was also given in \cite{Andrews}. Recently,
we generalized the gap conjecture to the Gross-Pitaevskii equation (GPE)
(or the nonlinear Schr\"{o}dinger equation with cubic repulsive interaction)
\cite{BaoRuan}.

The definition \eqref{def:FL_F} (or \eqref{def:FL_I}) is usually
called as the (classical) {\sl  fractional Laplacian} (see, for instance, \cite{Caffarelli1,Caffarelli2,Caffarelli3,Nezza,Servadei} and references therein).
In the literatures \cite{Barrios,Cabre,Capella,Tan} and references therein, there is another way -- {\sl local fractional Laplacian} denoted as $A^{(\alpha/2)}$ -- to define the fractional Laplacian  via the spectral decomposition of Laplacian \cite{Barrios,Cabre,Tan}. To be more specific, for a bounded domain $\Omega\subset\mathbb{R}^n$,
let $\lmd_\bm$ and $u_\bm$ ($\bm\in\mathbb{N}^d$) be the eigenvalues and corresponding eigenfunctions of the Laplacian operator $-\Delta$ on $\Omega$ with the homogeneous Dirichlet boundary condition, i.e. \eqref{eq:eigb} with $V_\bog(\bx)\equiv 0$. Then for any $\alpha\in(0,2)$ and $\phi(\bx)\in H_0^1(\Omega)$ with
\be
\phi(\bx)=\sum_{\bm\in\mathbb{N}^d} a_{\bm}\, u_\bm(\bx), \qquad \bx\in\overline{\Omega},
\ee
we define the operator $A^{(\alpha/2)}$ in the following way
\be\label{lfl11}
A^{(\alpha/2)}\phi(\bx)=\sum_{\bm\in\mathbb{N}^d} a_{\bm} \,\lmd_\bm^{\alpha/2} \, u_\bm(\bx), \qquad \bx\in\overline{\Omega}.
\ee
Comparison between the local fractional Laplacian operator $A^{(\alpha/2)}$
via \eqref{lfl11} and the classical fractional Laplacian operator $(-\Delta)^{\frac{\alpha}{2}}$ via \eqref{def:FL_F} (or \eqref{def:FL_I}) with zero extension
on $\Omega^c$ can be found in \cite{Servadei}. When $\alpha=2$, both definitions are the same. However, when $0<\alpha<2$, they are quite different.  One main difference is that the eigenfunctions of $A^{(\alpha/2)}$ is smooth inside $\Omega$ while the eigenfunctions of $(-\Delta)^{\frac{\alpha}{2}}$ is $C^{0,s}$ for some $s\in(0,1)$. And this H\"older regularity is optimal \cite{Servadei}. Then on the  bounded domain $\Omega$, for $\phi\in H_0^1(\Omega)$, one can define the \textbf{local fractional Schr\"{o}dinger operator} (local FSO) via the local fractional Laplacian as
\begin{equation}\label{lFSO112}
L_{\rm loc} \,\phi(\bx):=\left[A^{(\alpha/2)}+V_\bog(\bx)\right]\phi(\bx),
\qquad \bx\in \Omega.
\ee
Similarly, the fundamental gap of the local FSO \eqref{lFSO112}
is denoted as
\be\label{gaplfso}
\delta_{\rm loc}(\alpha):=\lambda_2(\alpha)-\lambda_1(\alpha)>0,
\qquad 0<\alpha\le 2.
\ee
where $0<\lambda_1(\alpha)<\lambda_2(\alpha)$ are the first two smallest
eigenvales of the local FSO \eqref{lFSO112}.

Due to the nonlocal property of the FSO, it is very challenging to study mathematically and numerically the eigenvalue
problem \eqref{eq:eig} \cite{Jeng}. In one dimension (1D),
some estimates and asymptotic approximations of eigenvalues of the FSO without potential (i.e. $V(\bx)\equiv 0$) have been derived (see \cite{Banuelos, Chen, Deblassie,Kwasnicki} and references therein).  It is noteworthy that Duo and Zhang \cite{Duo} introduced a finite difference scheme to solve the eigenvalue problems related to FSO in 1D. Nevertheless, to the best of our knowledge, not much is available about the numerical method for \eqref{eq:eig} in multi-dimensions.
The main purpose of this paper is to study
asymptotically and numerically the fundamental gap $\delta(\alpha)$
of the FSO \eqref{eq:eig} on bounded domains $\Omega$, i.e. the potential
$V(\bx)=+\infty$ for $\bx\in \Omega^c$, and $\delta_{\rm loc}(\alpha)$ of the local FSO \eqref{lFSO112}.  Based on our asymptotic results and extensive numerical results, we propose the following:

\textbf{Gap Conjecture I} {\em (Fundamental gaps of FSO on bounded domain with homogeneous Dirichlet boundary conditions) Suppose $\Omega$ is a bounded convex domain and  $V_\Omega(\mathbf{x})$ is convex and non-negative.

(i) For the fundamental gap of the local FSO \eqref{lFSO112},
 we have
\begin{align}\label{conj1}
\delta_{\rm loc}(\alpha)
&\ge \left\{\begin{array}{ll}
\frac{(2^{\alpha}-1)\pi^{\alpha}}{D^{\alpha}},\quad &n=1,\\
\frac{\alpha\pi^{\alpha}}{(n+2)^{1-\alpha/2}}\frac{d^{2-\alpha}}{D^2},\quad &n\ge2,\\
\end{array}\right. \qquad 0<\alpha\le 2.
\end{align}

(ii) For the fundamental gap of the  (classical) FSO \eqref{eq:fso}, we have
\begin{equation}\label{conj2}
\delta(\alpha)\ge\frac{\alpha\pi^{\alpha}}{(n+2)^{1-\alpha/2}}
\frac{d^{2-\alpha}}{D^2}, \qquad 0<\alpha\le 2.
\end{equation}
}
\medskip

In addition, we also propose a gap conjecture for  the FSO \eqref{eq:fso} in the whole space.

The paper is organised as follows. In Section \ref{sec:loc}, we study
asymptotically and numerically the fundamental gaps of the local FSO
\eqref{lFSO112} and formulate the gap conjecture \eqref{conj1}.
Similar results for the (classical) FSO \eqref{eq:fso} on bounded domains
with homogeneous Dirichlet boundary conditions are presented
in Section \ref{sec:loc_classic}.
In Section \ref{sec:global_classic}, we study
asymptotically and numerically the fundamental gaps of the FSO
\eqref{eq:fso} in the whole space and formulate a gap conjecture.
Again, similar results for the FSO \eqref{eq:fso} on bounded domains
with periodic boundary conditions are presented
in Section \ref{sec:periodic}.
Finally, some conclusions are drawn in Section \ref{sec:conclusion}.

\section{The fundamental gaps of the local FSO \eqref{lFSO112}} \label{sec:loc}
Consider the eigenvalue problem generated by the local FSO \eqref{lFSO112}
\begin{equation}\label{eq:eigl}
\begin{split}
&L_{\rm loc}\,\phi(\bx):=\left[A^{(\alpha/2)}+V_\bog(\bx)\right]\phi(\bx)=
\lambda\,\phi(\bx),\qquad \bx\in \Omega,\\
&\phi(\bx)=0, \qquad \bx\in \Gamma:=\partial\Omega.
\end{split}
\end{equation}
We will investigate asymptotically and numerically the first two smallest eigenvalues and their corresponding eigenfunctions of \eqref{eq:eigl}
and then formulate the gap conjecture \eqref{conj1}.

\subsection{Scaling property}
Introduce
\begin{equation}\label{scaxp}
\tbx=\frac{\bx}{D}, \ \tilde\Omega =\{\tbx \ | \ \bx=D\,\tbx\in \Omega\},
\ \tilde{V}_\botg(\tbx) =D^\alpha V_\bog(\bx)=D^\alpha V_\bog(D\tbx), \ \bx\in\Omega,
\end{equation}
and consider the re-scaled eigenvalue problem
\begin{equation}\label{eq:eiglr}
\begin{split}
&\tilde L_{\rm loc}\,\tilde\phi(\tbx):=\left[\tilde A^{(\alpha/2)}+\tilde {V}_\botg(\tbx)\right]\tilde\phi(\tbx)=
\tilde\lambda\,\tilde\phi(\tbx),\qquad \tbx\in \tilde\Omega,\\
&\tilde\phi(\tbx)=0, \qquad \tbx\in \tilde\Gamma:=\partial\tilde\Omega,
\end{split}
\end{equation}
where $\tilde A^{(\alpha/2)}$ is defined as \eqref{lfl11} with $\Omega$ replaced by $\tilde \Omega$, then we have

\begin{lemma}\label{lem:scale_loc}
Let $\lambda$ be an eigenvalue of \eqref{eq:eigl} and $\phi:=\phi(\bx)$ is the corresponding eigenfunction, then $\tilde \lambda =D^\alpha \lambda$
is an eigenvalue of \eqref{eq:eiglr} and  $\tilde \phi:=\tilde \phi(\tbx) = \phi(D\tbx)=\phi(\bx)$ is the corresponding eigenfunction, which immediately imply the scaling property on the fundamental gap $\delta_{\rm loc}(\alpha)$  of \eqref{eq:eigl} as
\be
\delta_{\rm loc}(\alpha) = \frac{\tdelta_{\rm loc}(\alpha)}{D^{\alpha}}, \qquad
0<\alpha\le 2,
\ee
where $\tdelta_{\rm loc}(\alpha)$ is the fundamental gap of
\eqref{eq:eiglr} with the diameter of $\tilde \Omega$ as $1$.
\end{lemma}

\begin{proof}
Assume $\lambda_\bm$ be an eigenvalue of \eqref{eq:eigb} with
$V_\Omega(\bx)\equiv 0$ and $u_\bm(\bx)$ be the corresponding
eigenfunction, i.e. $u_\bm(\bx)\in H_0^1(\Omega)$ satisfies
\be
-\Delta u_\bm(\bx)=\lmd_\bm u_\bm(\bx),  \quad \bx\in\Omega.
\ee
It is easy to see that
\be
-\Delta \tu_\bm(\tbx)=D^{2}\lmd_\bm \tu_\bm(\tbx)=\tilde\lmd_\bm \tu_\bm(\tbx), \quad \tbx\in\tOmega,
\ee
where $\tilde\lmd_\bm=D^{2}\lmd_\bm$. Then
for any $\tilde \phi(\tbx)\in H_0^1(\tilde \Omega)$, recalling the definition of the local fractional Laplacian \eqref{lfl11}, we get
\bea\label{Aaph11}
A^{(\alpha/2)}\phi(\bx)&=&\sum_{\bm\in\mathbb{N}^d} a_{\bm} \,\lmd_\bm^{\alpha/2} \, u_\bm(\bx)=D^{-\alpha} \sum_{\bm\in\mathbb{N}^d} a_{\bm} \,(D^2\lmd_\bm)^{\alpha/2} \,  u_\bm(\bx)\nonumber\\
&=& D^{-\alpha} \sum_{\bm\in\mathbb{N}^d} a_{\bm} \,(\tilde\lmd_\bm)^{\alpha/2} \, \tilde u_\bm(\tbx)=
D^{-\alpha} \tilde A^{(\alpha/2)}\tilde \phi(\tbx), \qquad \bx\in \Omega.
\eea
Plugging \eqref{Aaph11} into \eqref{eq:eigl}, noticing
\eqref{eq:eiglr}, we get
\bea
\lambda\,\tphi(\tbx)&=&\lambda\,\phi(\bx)=\left[A^{(\alpha/2)}+V_\bog(\bx)\right]\phi(\bx)=\left[D^{-\alpha} \tilde A^{(\alpha/2)}+V_\bog(D\tbx)\right]\tphi(\tbx)\nonumber\\
&=&D^{-\alpha}
\left[ \tilde A^{(\alpha/2)}+D^\alpha V_\bog(D\tbx)\right]
\tphi(\tbx)=D^{-\alpha}\left[ \tilde A^{(\alpha/2)}+\tilde V_\botg(\tbx)\right]
\tilde \phi(\tbx),
\eea
where $\bx\in \Omega$ and $\tbx\in \tilde \Omega$,
which immediately implies that $\tilde \phi(\tbx)$ is an eigenfunction of the operator $\tilde A^{(\alpha/2)}+\tilde V_\botg(\tbx)$ with the eigenvalue $\tilde \lambda=D^\alpha \lambda$.
\end{proof}

  From this scaling property, in our asymptotic analysis and numerical
simulation, we need only consider $\Omega$ whose diameter is $1$ in \eqref{eq:eigl}.

\subsection{Asymptotic results for simple geometry}
Take $\Omega=\prod_{j=1}^n(0,L_j)$ and $V_\bog(\bx)\equiv 0$ in \eqref{eq:eigl}.
Without loss of generality, we assume $L_1\ge L_2 \ge \ldots \ge L_n>0$.
In this case, the first two smallest eigenvalues and their corresponding
eigenfunctions of $-\Delta$ can be chosen explicitly as \cite{BC,BL}
\be\begin{split}
&\lambda_1=\sum_{j=1}^n\frac{\pi^2}{L_j^2}, \qquad  \phi_1(\bx)=2^{n/2}\prod_{j=1}^n \sin\left(\frac{\pi x_j}{L_j}\right),\qquad \bx\in \overline\Omega, \\
&\lambda_2=\frac{4\pi^2}{L_1^2}+
\sum_{j=2}^n\frac{\pi^2}{L_j^2}, \qquad  \phi_2(\bx)=2^{n/2}\sin\left(\frac{2\pi x_j}{L_1}\right)\prod_{j=2}^n \sin\left(\frac{\pi x_j}{L_j}\right).\\
\end{split}
\ee
By using the definition of the local FSO, we can obtain
the first two smallest eigenvalues and the fundamental gap in this case as
\be\label{lmd123}
\begin{split}
&\lambda_1(\alpha)=\left(\sum_{j=1}^n\frac{\pi^2}{L_j^2}\right)^{\alpha/2},
\qquad \lambda_2(\alpha)=\left(\frac{4\pi^2}{L_1^2}+
\sum_{j=2}^n\frac{\pi^2}{L_j^2}\right)^{\alpha/2}, \\
&\delta_{\rm loc}(\alpha)=\lambda_2(\alpha)-\lambda_1(\alpha),
\qquad 0<\alpha\le 2.
\end{split}
\ee

Formally, when $n\ge2$, let $L_2,\dots, L_n\to0^+$ in \eqref{lmd123}, we have the diameter $D\to L_1$ and $\nu:=\left(\sum_{j=2}^n\frac{\pi^2}{L_j^2}\right)^{1/2} \to+\infty$. When $\alpha=2$,
\be\label{gap2loc}
\delta_{\rm loc}(2)=\frac{3\pi^2}{L_1^2}=\frac{3\pi^2}{D^2}>0,
\ee
i.e. the fundamental gap is independent of the shape of the geometry and it
only depends the diameter $D$ of $\Omega$.
On the contrary, when $0<\alpha<2$,
\bea\label{gapl2loc}
\delta_{\rm loc}(\alpha)&=&\left(\nu^2+\frac{4\pi^2}{L_1^2}\right)^{\alpha/2}
-\left(\nu^2+\frac{\pi^2}{L_1^2}\right)^{\alpha/2}\nonumber\\
&=&\nu^{\alpha}\left[
\left(1+\frac{4\pi^2}{\nu^2 L_1^2}\right)^{\alpha/2}
-\left(1+\frac{\pi^2}{\nu^2 L_1^2}\right)^{\alpha/2}\right]\nonumber\\
&=&\frac{\alpha\nu^{\alpha}}{2} \; \frac{1}{(1+\xi)^{1-\alpha/2}}\;\frac{3\pi^2}{\nu^2 L_1^2}\nonumber\\
&\le&\frac{3\alpha \pi^2}{L_1^2 \,\nu^{2-\alpha}}\to 0^+, \qquad 0<\alpha<2,
\eea
where $0<\xi \in [\pi^2/(\nu^2 L_1^2), 3\pi^2/(\nu^2 L_1^2)]$.
In this case, the lower bound of the fundamental gap depends not only
on the diameter $D$ of $\Omega$ but also another geometry quantity.
By looking carefully  at \eqref{gapl2loc}, we find that
the diameter of the largest inscribed ball of $\Omega$, i.e. $d$,
seems to be a good choice since its ratio with the diameter $D$ can be used to measure whether the domain degenerates from $n$ dimensions to lower dimensions. Based on these observation, we have the
 following lemma.

\smallskip

\begin{lemma}\label{lem:gap_box}
For $\Omega=\prod_{j=1}^n(0,L_j)$ satisfying
$L_1\ge L_2 \ge \ldots \ge L_n>0$
and $V_\bog(\bx)\equiv 0$ in \eqref{eq:eigl},
we have the following lower bound of the fundamental gaps
of the local FSO in \eqref{eq:eigl}
\be\label{gaplc234}
\delta_{\rm loc}(\alpha)\ge\left\{
\begin{array}{ll}
\frac{(2^{\alpha}-1)\pi^{\alpha}}{D^{\alpha}},\quad &n=1,\\
\frac{\alpha\pi^{\alpha}}{(n+2)^{1-\alpha/2}}\frac{d^{2-\alpha}}{D^2},\quad &n\ge2,\\
\end{array}\right. \qquad 0<\alpha\le 2,
\ee
where $D=\sqrt{\sum_{j=1}^nL_j^2}$ is the diameter of $\Omega$ and $d=L_n$ is the diameter of the largest inscribed ball in $\Omega$.
\end{lemma}

\begin{proof} When $n=1$, noticing $D=L_1$ and \eqref{lmd123} with $n=1$, we have
\bea
\delta_{\rm loc}(\alpha)&=&\lambda_2(\alpha)-\lambda_1(\alpha)
=\left(\frac{4\pi^2}{L_1^2}\right)^{\alpha/2}-
\left(\frac{\pi^2}{L_1^2}\right)^{\alpha/2}\nonumber\\
&=&\frac{(2^{\alpha}-1)
\pi^{\alpha}}{L_1^{\alpha}}=\frac{(2^{\alpha}-1)\pi^{\alpha}}{D^{\alpha}},
\qquad 0<\alpha\le 2.
\eea
Thus \eqref{gaplc234} is proved when $n=1$.

When $n\ge2$, noticing \eqref{lmd123},  we have
\be\label{proof:box1}
\delta_{\rm loc}(\alpha)=\left(\frac{4\pi^2}{L_1^2}+
\sum_{j=2}^n\frac{\pi^2}{L_j^2}\right)^{\alpha/2}-
\left(\frac{\pi^2}{L_1^2}+\sum_{j=2}^n\frac{\pi^2}{L_j^2}\right)^{\alpha/2},\qquad 0<\alpha\le 2.
\ee
We will first prove that
\be\label{proof:box}
\delta_{\rm loc}(\alpha)\ge
\left(\frac{4\pi^2}{D^2}+\frac{(n-1)\pi^2}{d^2}\right)^{\alpha/2}
-\left(\frac{\pi^2}{D^2}+\frac{(n-1)\pi^2}{d^2}\right)^{\alpha/2},\qquad 0<\alpha\le 2.
\ee
In order to do so, we consider two functions
\be\label{app:func}
\begin{split}
&f(x;C)=\left(\frac{4\pi^2}{x^2}+C^2\right)^{\alpha/2}-
\left(\frac{\pi^2}{x^2}+C^2\right)^{\alpha/2},
\qquad x>0, \\
&g(x;A,B)=\left(x+A+B\right)^{\alpha/2}-\left(x+A\right)^{\alpha/2},\\
\end{split}
\ee
where $0<\alpha\le 2$, $C\in\mathbb{R}$ and $A,B\ge0$. A direct computation shows that $\frac{d}{dx}f(x;C)\le0$ and $\frac{d}{dx}g(x;A,B)\le0$ for $x>0$, which means that $f(x;C)$ and $g(x;A,B)$ are monotonically decreasing functions.
When $n=2$, it is easy to check that $d=L_2$ and $D\ge L_1$. Noticing
$f(D;\pi/d)\le f(L_1;\pi/d)$, we immediately obtain \eqref{proof:box} when $n=2$.
When $n=3$, noting $d=L_3\le L_2\le L_1$ and $D\ge L_1\ge L_2\ge L_3$, we get
\bea
\delta_{\rm loc}(\alpha)&=&\left(\frac{\pi^2}{L_2^2}+\frac{\pi^2}{L_1^2}
+\frac{\pi^2}{L_3^2}+\frac{3\pi^2}{L_1^2}\right)^{\alpha/2}-
\left(\frac{\pi^2}{L_2^2}+\frac{\pi^2}{L_1^2}+\frac{\pi^2}{L_3^2}
\right)^{\alpha/2}\nonumber\\
&=&g\left(\frac{\pi^2}{L_2^2};\frac{\pi^2}{L_1^2}+
\frac{\pi^2}{L_3^2},\frac{3\pi^2}{L_1^2}\right)
\ge g\left(\frac{\pi^2}{d^2};\frac{\pi^2}{L_1^2}+\frac{\pi^2}{L_3^2},
\frac{3\pi^2}{L_1^2}\right)\nonumber\\
&=&\left(\frac{4\pi^2}{L_1^2}+\frac{2\pi^2}{d^2}\right)^{\alpha/2}
-\left(\frac{\pi^2}{L_1^2}+\frac{2\pi^2}{d^2}\right)^{\alpha/2}
=f\left(L_1;\frac{\sqrt{2}\pi}{d}\right)
\ge f\left(D;\frac{\sqrt{2}\pi}{d}\right)\nonumber\\
&=&\left(\frac{4\pi^2}{D^2}+\frac{2\pi^2}{d^2}\right)^{\alpha/2}
-\left(\frac{\pi^2}{D^2}+\frac{2\pi^2}{d^2}\right)^{\alpha/2},\qquad
0<\alpha\le 2,
\eea
which proves \eqref{proof:box} when $n=3$.

 When $n\ge2$, noting \eqref{proof:box}, we get
\bea\label{mean345}
\delta_{\rm loc}(\alpha)&\ge&\left(\frac{4\pi^2}{D^2}+\frac{(n-1)
\pi^2}{d^2}\right)^{\alpha/2}-\left(\frac{\pi^2}{D^2}+
\frac{(n-1)\pi^2}{d^2}\right)^{\alpha/2}\nonumber\\
&=&\left(\frac{(n-1)\pi^2}{d^2}\right)^{\alpha/2}
\left[\left(1+\frac{3}{n-1}\left(\frac{d}{D}\right)^2\right)^{\alpha/2}
-\left(1+\frac{1}{n-1}\left(\frac{d}{D}\right)^2\right)^{\alpha/2}
\right]\nonumber\\
&=&\left(\frac{(n-1)\pi^2}{d^2}\right)^{\alpha/2}\;\frac{\alpha}
{n-1}\;\frac{1}{(1+\xi)^{1-\alpha/2}}\;\left(\frac{d}{D}\right)^2, \qquad 0<\alpha\le 2,
\eea
where the last equation is due to the mean value theorem with
$\xi\in[\frac{1}{n-1}\left(\frac{d}{D}\right)^2,\frac{3}{n-1}
\left(\frac{d}{D}\right)^2]\subset[0,\frac{3}{n-1}]$.
Noting that
$\frac{1}{(1+\xi)^{1-\alpha/2}}$ is a decreasing function when
$\xi\ge0$ and taking $\xi=\frac{3}{n-1}$ in \eqref{mean345}, we obtain the result \eqref{lmd123}  when $n\ge2$ by
\eqref{proof:box} and \eqref{mean345}.
\end{proof}

\smallskip

\begin{remark}
When $n=1$, noting $d=D=L_1$, we have
\be\label{gaploc456}
\delta_{\rm loc}(\alpha)\ge \frac{(2^{\alpha}-1)\pi^{\alpha}}{D^{\alpha}}\ge \frac{\alpha\pi^{\alpha}}{3^{1-\alpha/2}}\frac{d^{2-\alpha}}{D^2}, \qquad
0<\alpha\le 2.
\ee
Combing \eqref{gaploc456} and \eqref{gaplc234} with $n\ge2$,
we have a unified local bound of the fundamental gap of the local
FSO as
\be\label{lgap765}
\delta_{\rm loc}(\alpha)\ge
\frac{\alpha\pi^{\alpha}}{(n+2)^{1-\alpha/2}}\frac{d^{2-\alpha}}{D^2},
\qquad 0<\alpha\le 2, \quad 1\le n\le 3.
\ee
Of course, the lower bound is not sharp when $n=1$.
\end{remark}


\subsection{Numerical results for complicated geometry and/or general potentials}
When $\Omega$ is  a complicated domain and/or $V_\bog(\bx)\ne 0$ in
\eqref{eq:eigl}, it is not  generally possible to  find the first two smallest eigenvalues explicitly. However, we can always compute numerically
the first two smallest eigenvalues and their gap of  \eqref{eq:eigl} under a given bounded convex domain $\Omega$ and
a convex real-valued function $V_\bog(\bx)$.
Some numerical methods for local fractional Laplacian have been proposed in the literatures, e.g.,  a matrix representation of local fractional Laplacian operator based on a finite difference method is presented in \cite{ilic05,ilic06}; Fourier spectral methods for  solving local fractional Laplacian  can be found in e.g., \cite{bueno14,Ainsw17}. Recently, Sheng et al. \cite{Sheng18} proposed a Fourierization of Legendre-Galerkin method for PDEs with local fractional Laplacian. The method retains the simplicity of Fourier method but is applicable to problems with non-periodic boundary conditions.   In this paper, we adopt this method  to numerically compute  the first two smallest eigenvalues of \eqref{eq:eigl}.

Fig. \ref{fig:box_1d_gap_loc} shows the numerical results
on the fundamental gap $\delta_{\rm loc}(\alpha)$ of \eqref{eq:eigl}
when $n=1$, $\Omega=(0,1)$  and different external potentials $V_\bog(x)$.
Fig. \ref{fig:box_2d_gap_loc}  shows similar results
when $n=2$, $V_\bog(x,y)=(x^2+y^2)/2$ and  $\Omega=(0,d)\times(0,\sqrt{1-d^2})$
with different $0<d< 1$;
and  $V_\bog(x,y)\equiv 0$ and $\Omega=\{(x,y)\ | \ x^2+y^2/d^2 <1\}$ with
different $0<d\le 1$.

\begin{figure}[htbp]
\centerline{\psfig{figure=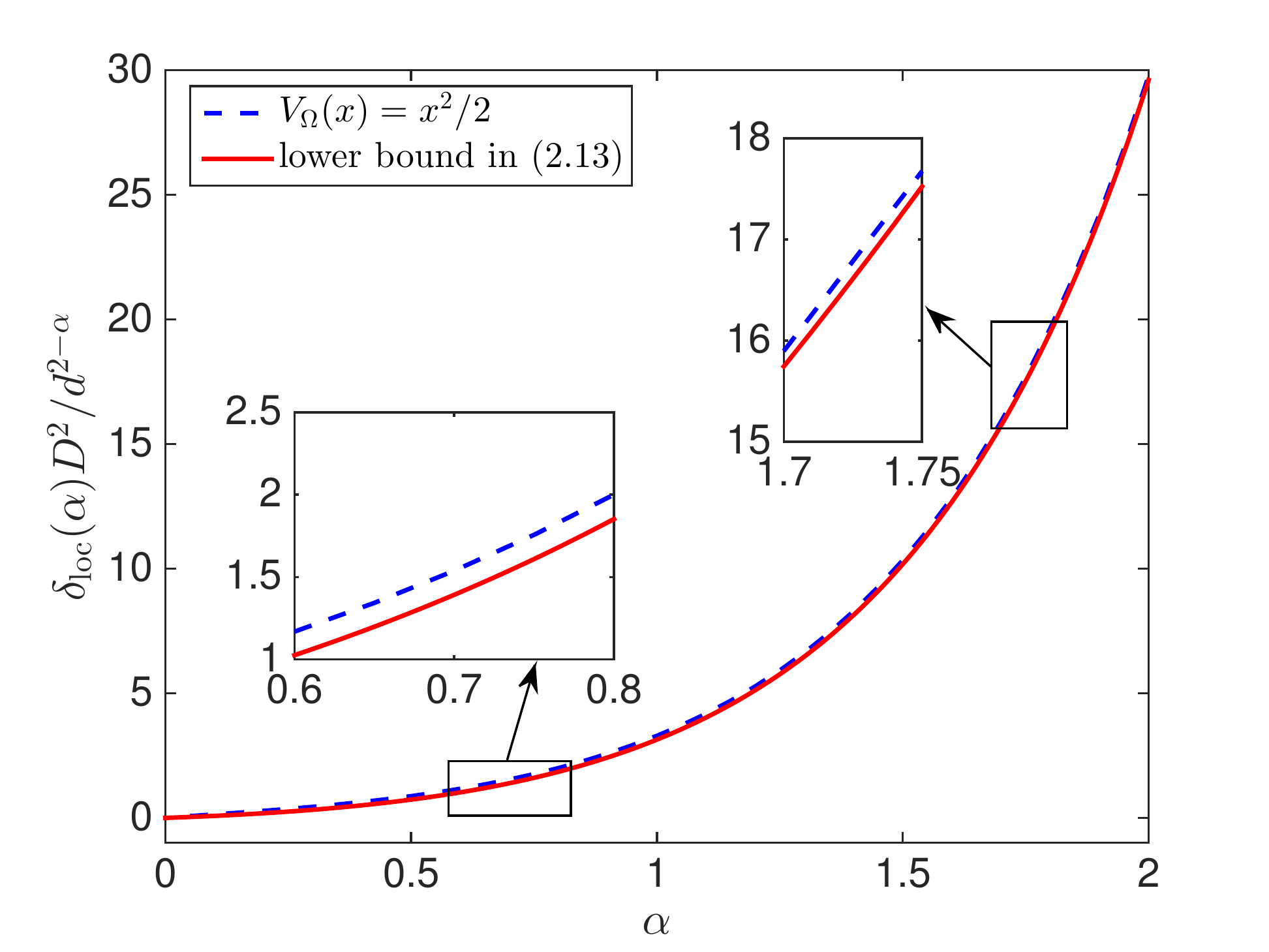,height=6.5cm,width=7cm,angle=0}
\psfig{figure=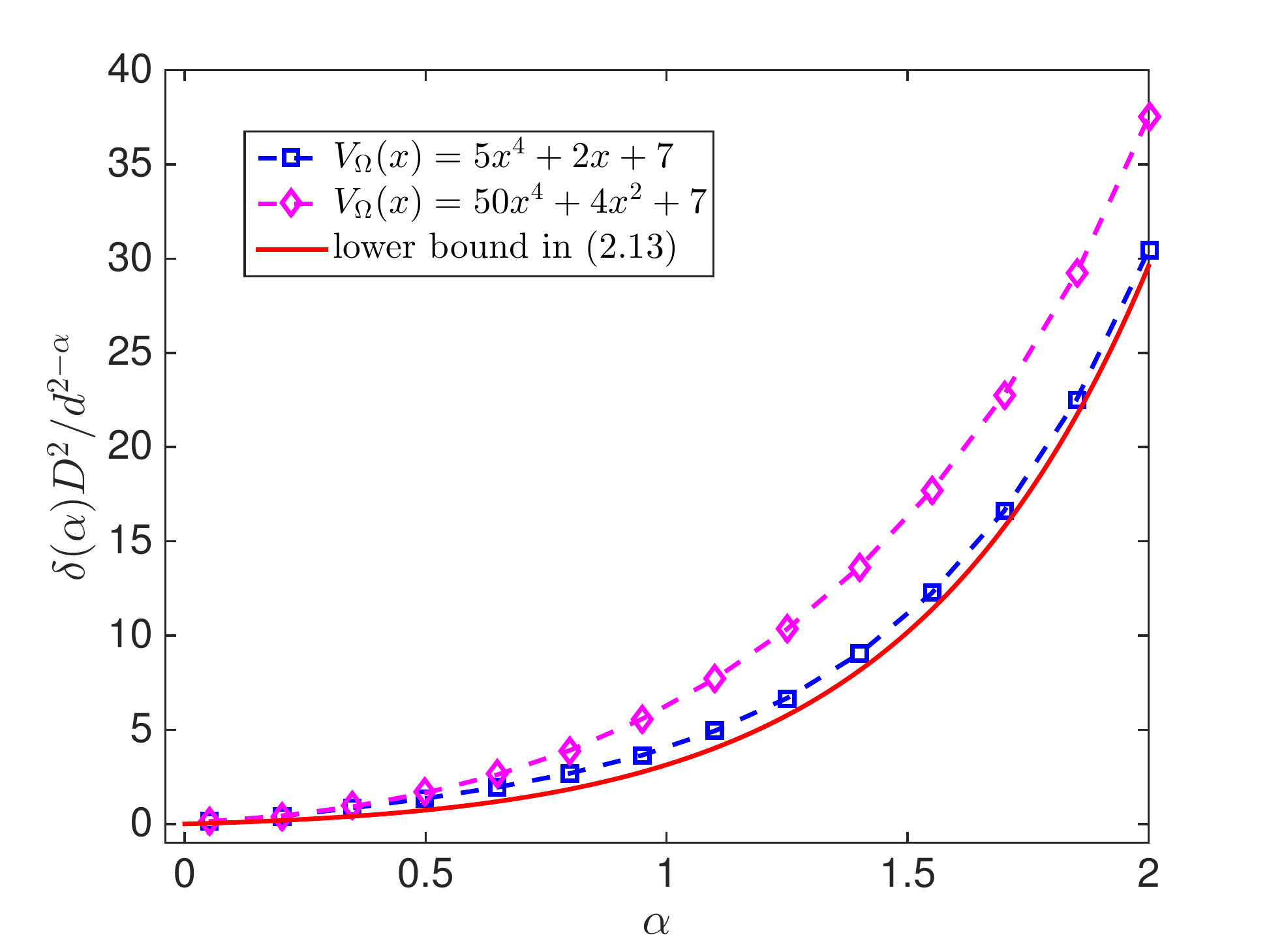,height=6.5cm,width=7cm,angle=0}}
\caption{Comparison of the lower bound in \eqref{gaplc234} (solid line)
and numerical results (dash lines) for the fundamental gap $\delta_{\rm loc}(\alpha)$ of the  local FSO \eqref{eq:eigl} with $n=1$, $\Omega=(0,1)$ and $V_\bog(x)=x^2/2$ (left) or other different convex potentials (right).}
\label{fig:box_1d_gap_loc}
\end{figure}

\begin{figure}[htbp]
\centerline{\psfig{figure=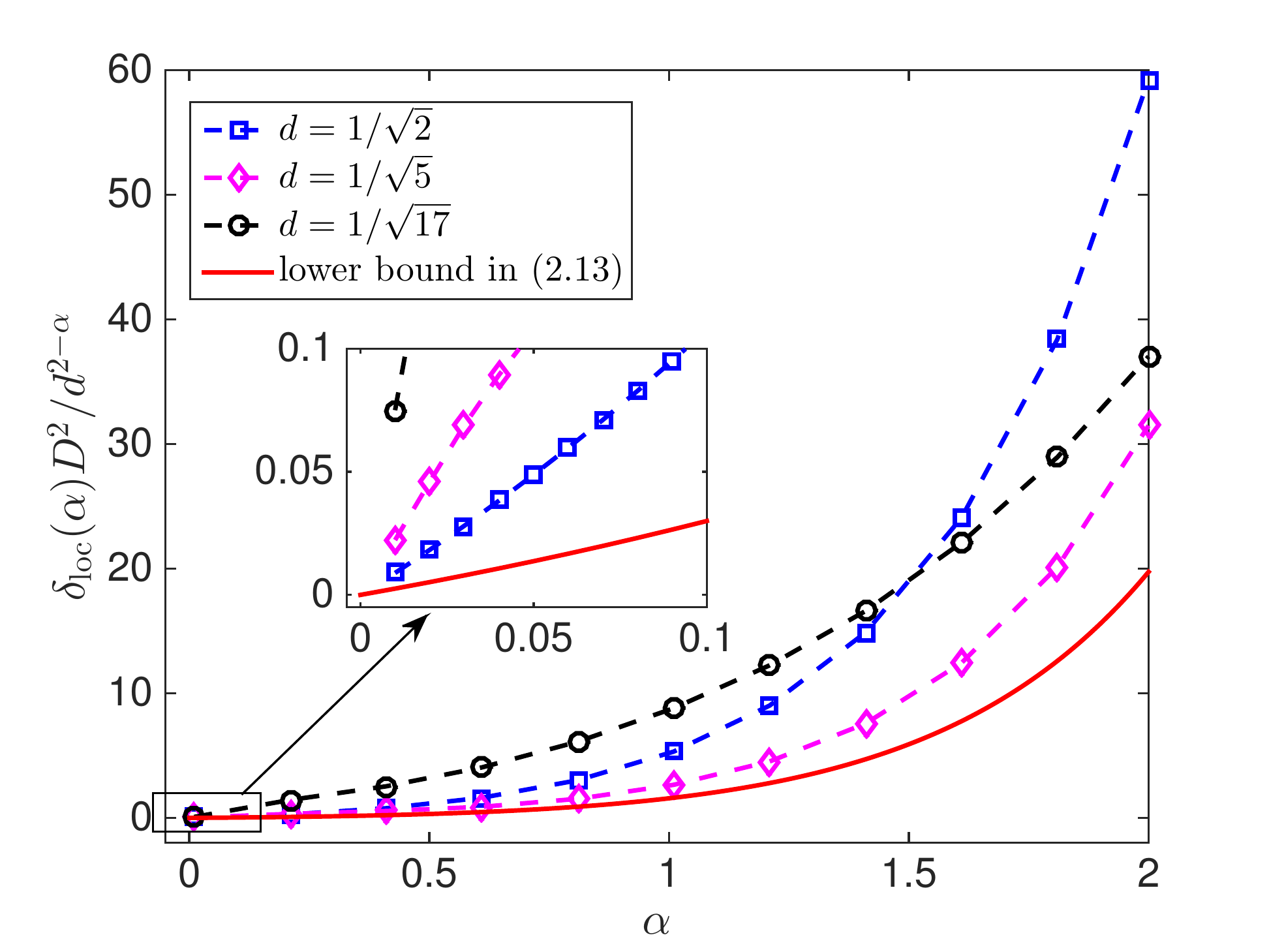,height=6.5cm,width=7cm,angle=0}
\psfig{figure=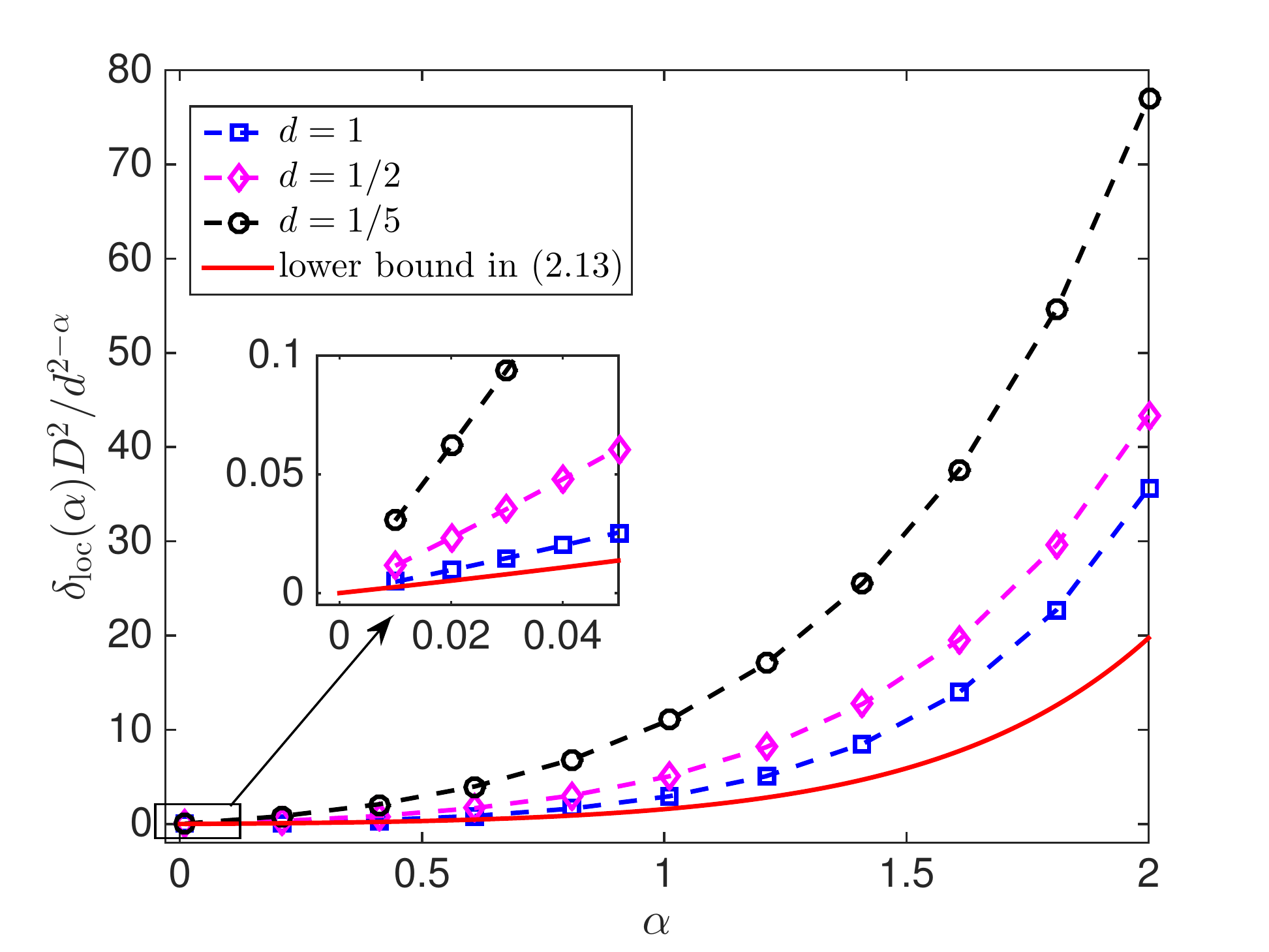,height=6.5cm,width=7cm,angle=0}}
\caption{Comparison of the lower bound in \eqref{gaplc234} (solid line)
and numerical results (dash lines) for the fundamental gap $\delta_{\rm loc}(\alpha)$ of the  local FSO \eqref{eq:eigl} with $n=2$, $V_\bog(x,y)=(x^2+y^2)/2$ and  $\Omega=(0,d)\times(0,\sqrt{1-d^2})$
under different $0<d<1$ (left); and $V_\bog(x,y)\equiv 0$ and $\Omega=\{(x,y)\ | \ x^2+y^2/d^2 <1\}$ with
different $0<d\le 1$ (right).}
\label{fig:box_2d_gap_loc}
\end{figure}

Based on our asymptotic results in the previous subsection and numerical results in  Figs. \ref{fig:box_1d_gap_loc}\&\ref{fig:box_2d_gap_loc}
as well as extensive more numerical results which draw similar conclusion and thus are not shown here for brevity,
we are confident to formulate the gap conjecture
\eqref{conj1} for the local FSO \eqref{lFSO112}.

\section{The fundamental gaps of the FSO \eqref{eq:fso} on bounded domains}\label{sec:loc_classic}
Consider the eigenvalue problem generated by the FSO \eqref{eq:fso}
\begin{equation}\label{eq:eigfso}
\begin{split}
&L_{\rm FSO}\,\phi(\bx):=\left[(-\Delta)^{\frac{\alpha}{2}}+V_\bog(\bx)\right]\phi(\bx)=
E\,\phi(\bx),\qquad \bx\in \Omega,\\
&\phi(\bx)=0, \qquad \bx\in \Omega^c.
\end{split}
\end{equation}
In fact, if $\phi(\bx)$ is an eigenfunction normalized as
\be\label{normbd12}
\|\phi\|^2:=\int_{\mathbb{R}^n}\,|u(\bx)|^2\,d\bx=
\int_{\Omega}\,|u(\bx)|^2\,d\bx=1,
\ee
then the corresponding eigenvalue $E>0$
can also be computed as
\bea\label{Ebd12}
E:&=&E^{(\alpha)}(\phi)=\int_{\Omega}\,
\left[\phi(\bx)^*(-\Delta)^{\alpha/2}\phi(\bx) +V_\bog(\bx)|\phi(\bx)|^2\right]\,d\bx\nonumber \\
&=&\int_{{\mathbb R}^n}|\bk|^\alpha |\hat{\phi}(\bk)|^2\,d\bk+\int_{\Omega}\,
 V_\bog(\bx)|\phi(\bx)|^2\,d\bx,
\eea
where $\hat{\phi}:=\hat{\phi}(\bk)$ is the Fourier transform of $\phi:=\phi(\bx)$.
We will investigate asymptotically and numerically the first two smallest eigenvalues and their corresponding eigenfunctions of \eqref{eq:eigfso}
and then formulate the gap conjecture \eqref{conj2}.

\subsection{Scaling property}
Under the transformation \eqref{scaxp},
consider the re-scaled eigenvalue problem
\begin{equation}\label{eq:eigfsor}
\begin{split}
&\tilde L_{\rm FSO}\,\tilde \phi(\tbx):=\left[(-\Delta)^{\frac{\alpha}{2}}+\tilde V_\botg(\tbx)\right]\tilde \phi(\tbx)=
\tilde E\,\tilde \phi(\tbx),\qquad \tbx\in \tilde\Omega,\\
&\tilde\phi(\tbx)=0, \qquad \tbx\in \tilde\Omega^c.
\end{split}
\end{equation}
Then we have

\begin{lemma}\label{lem:scalefso}
Let $E$ be an eigenvalue of \eqref{eq:eigfso} and $\phi:=\phi(\bx)$ is the corresponding eigenfunction, then $\tilde E =D^\alpha E$
is an eigenvalue of \eqref{eq:eigfsor} and  $\tilde \phi:=\tilde \phi(\tbx) = \phi(D\tbx)=\phi(\bx)$ is the corresponding eigenfunction, which immediately imply the scaling property on the fundamental gap $\delta(\alpha)$  of \eqref{eq:eigfso} as
\be
\delta(\alpha) = \frac{\tdelta(\alpha)}{D^{\alpha}}, \qquad
0<\alpha\le 2,
\ee
where $\tdelta(\alpha)$ is the fundamental gap of
\eqref{eq:eigfsor} with the diameter of $\tilde \Omega$ as $1$.
\end{lemma}

\begin{proof}
From \eqref{def:FL_I}, a direct computation implies the scaling property of the  fractional Laplacian operator
\bea\label{dtapsc}
(-\Delta)^{\alpha/2}\phi(\bx)&=&C_{n,\alpha}
\int_{\mathbb{R}^n}\frac{\phi(\bx)-\phi(\by)}{|\bx-\by|^{n+\alpha}}\,d\by
= C_{n,\alpha}\int_{\mathbb{R}^n}\frac{\phi(D\tbx)
-\phi(D\tby)}{|D\tbx-D\tby|^{n+\alpha}}\,D^{n}\,{d\tby} \nonumber\\
&=&D^{-\alpha} C_{n,\alpha}\int_{\mathbb{R}^n}\frac{\tilde\phi(\tbx)-
\tilde\phi(\tby)}{|\tbx-\tby|^{n+\alpha}}\,d\tby
=D^{-\alpha}\,(-\Delta)^{\alpha/2}\tilde\phi(\tbx), \quad \bx\in\Omega,
\ \tbx\in \tilde\Omega.
\eea
Noticing
\be
\phi(\bx)=0, \quad \bx\in \Omega^c \qquad \Longleftrightarrow \qquad \tilde\phi(\tbx)=0, \quad \tbx\in \tilde\Omega^c.
\ee
Substituting \eqref{dtapsc} into \eqref{eq:eigfso}, noting \eqref{eq:eigfsor}, we get
\bea
E\,\tphi(\tbx)&=&E\,\phi(\bx)=\left[(-\Delta)^{\frac{\alpha}{2}}+V_\bog(\bx)\right]
\phi(\bx)=\left[D^{-\alpha}\,(-\Delta)^{\frac{\alpha}{2}} +V_\bog(D\tbx)\right]\tphi(\tbx)\nonumber\\
&=&D^{-\alpha}
\left[(-\Delta)^{\frac{\alpha}{2}} +D^\alpha V_\bog(D\tbx)\right]
\tphi(\tbx)
=D^{-\alpha}
\left[ (-\Delta)^{\frac{\alpha}{2}}+\tilde V_\botg(\tbx)\right]
\tilde \phi(\tbx),
\eea
where $\bx\in \Omega$ and $\tbx\in \tilde \Omega$,
which immediately implies that $\tilde \phi(\tbx)$ is an eigenfunction of the operator $(-\Delta)^{\frac{\alpha}{2}}+\tilde V_\botg(\tbx)$ with the eigenvalue $\tilde E=D^\alpha E$.
\end{proof}

\subsection{Asymptotic results when $0\le 2-\alpha\ll1$}\label{subsec:box_potential}
For the fundamental gap $\delta(\alpha)$ of the FSO \eqref{eq:fso}
in 1D with box potential, we have


\smallskip

\begin{lemma}
Taken $n=1$, $\Omega=(0,1)$ and $V_\bog(x)\equiv 0$ for $x\in \Omega$ in \eqref{eq:eigfso}, when $0\le \varepsilon:=2-\alpha\ll1$, we
have
\begin{align}\label{asym:box_1d}
\delta(\alpha)\approx& -\frac{2\pi^2}{\Gamma(4-\alpha)}
\sec(\alpha\pi/2)\left[4\,\,_1F_2(2;2-\alpha/2,5/2-\alpha/2;
-\pi^2)\right.\nonumber\\
&\left.+\,\,_1F_2(2;2-\alpha/2,5/2-\alpha/2;-\pi^2/4)\right],
\end{align}
where $_pF_q(a_1,\dots,a_p;b_1,\dots,b_q;z)$ is the generalized hypergeometric function defined as \cite{Askey,Daalhuis}
\be\label{def:pFq}
_pF_q(a_1,\dots,a_p;b_1,\dots,b_q;z)=\sum_{k=0}^{\infty}
\frac{(a_1)_k\dots(a_p)_k}{(b_1)_k\dots(b_q)_k}\frac{z^k}{k!},
\ee
with $(a)_0=1$ and $(a)_k=a(a+1)\dots(a+k-1)$.
\end{lemma}

\smallskip

\begin{proof}
For $n=1$, $\Omega=(0,1)$ and $V_\bog(x)\equiv 0$ in \eqref{eq:eigfso},
when $\alpha=2$, the first two smallest eigenvalues and their corresponding
normalized eigenfunctions can be given as \cite{BC,BL}
\be
E_l(2)=l^2\pi^2, \qquad \phi_l(x)=\left\{
\begin{array}{ll}
\sqrt{2}\sin(l\pi x), &\text{ if } x\in(0,1)\\
0, &\text{ otherwise,}\\
\end{array}
\right. \quad l=1,2.
\ee
The Fourier transform of $\phi_l(x)$ ($l=1,2$) can be computed as
\be\label{Ebd1289}
\hat{\phi}_l(k)=\frac{\sqrt{2}l\pi((-1)^le^{-ik}-1)}{k^2-l^2\pi^2},
\qquad k\in{\mathbb R}.
\ee
It is worth noticing that $k=\pm l\pi$ are not singular points of $\hat{\phi}_l(k)$. In fact, we have that $\lim_{k\to l\pi}\hat{\phi}_l(k)=-i/2$ and $\lim_{k\to -l\pi}\hat{\phi}_l(k)=i/2$.
When $\alpha$ satisfies $0\le 2-\alpha\ll 1$,
the two normalized eigenfunctions $\phi_l^{(\alpha)}(x)$ ($l=1,2$) corresponding to the first two smallest eigenvalues
of \eqref{eq:eigfso} can be well approximated by $\phi_l(x)$ ($l=1,2$), respectively, i.e.
\be\label{efbda12}
\phi_1^{(\alpha)}(x)\approx \phi_1(x),
\qquad \phi_2^{(\alpha)}(x)\approx \phi_2(x), \qquad
x\in {\mathbb R}.
\ee
Substituting \eqref{efbda12} into \eqref{Ebd12}, noting \eqref{Ebd1289},
we can obtain
the approximations of the first two smallest eigenvalues $E_l(\alpha)$
($l=1,2$) as
\bea\label{Eka456}
E_l(\alpha)&=&E^{(\alpha)}(\phi_l^{(\alpha)})
\approx E^{(\alpha)}(\phi_l)=\int_{0}^1\,\phi_l(x)\, (-\Delta)^{\alpha/2}\phi_l(x)\,dx\nonumber\\
&=&\frac{1}{2\pi}\int_{\mathbb{R}}\,|k|^{\alpha}|\hat{\phi}_l(k)|^2\,dk
=2l^2\pi\int_{\mathbb{R}}\frac{1-(-1)^l\cos(k)}{(k^2-l^2\pi^2)^2}
|k|^{\alpha}\,dk, \quad l=1,2.
\eea
Combining \eqref{def:pFq} and \eqref{Eka456}, we obtain
\be\label{E1E2a}
\begin{split}
&E_1(\alpha)\approx\frac{2^{\alpha-2}\pi^{\frac{5}{2}}
\sec(\alpha\pi/2)}{\Gamma\left(2-\frac{\alpha}{2}\right)
\Gamma\left(\frac{5-\alpha}{2}\right)}\, \,_1F_2(2;2-\alpha/2,5/2-\alpha/2;-\pi^2/4),\\
&E_2(\alpha)\approx\frac{2^{\alpha}\pi^{\frac{5}{2}}
\sec(\alpha\pi/2)}{\Gamma\left(2-\frac{\alpha}{2}\right)
\Gamma\left(\frac{5-\alpha}{2}\right)} \,\,_1F_2(2;2-\alpha/2,5/2-\alpha/2;-\pi^2).
\end{split}
\ee
Plugging \eqref{E1E2a} into \eqref{gapfso}, we obtain \eqref{asym:box_1d} immediately and
the proof is completed.
\end{proof}


\smallskip

Similarly, taken $n=2$, $\Omega=(0,1)\times(0,L)$ with $0<L\le 1$ and $V_\bog(\bx)\equiv 0$ for $\bx\in \Omega$ in \eqref{eq:eigfso}, when $0<\varepsilon:=2-\alpha\ll1$, we
have
\be\label{asym:box_2d1}
\begin{split}
&E_1(\alpha)\approx \frac{\pi^2}{L^3}\iint_{\mathbb{R}^2}\, (k_1^2+k_2^2)^{\alpha/2} \; \frac{2+2\cos(k_1)}{(k_1^2-\pi^2)^2}\;\frac{2+2\cos(k_2L)}
{(k_2^2-\pi^2/L^2)^2}\,dk_1dk_2,\\
&E_2(\alpha)\approx \frac{4\pi^2}{L^3}\iint_{\mathbb{R}^2}\, (k_1^2+k_2^2)^{\alpha/2} \; \frac{2-2\cos(k_1)}{(k_1^2-4\pi^2)^2}\;\frac{2+2\cos(k_2L)}
{(k_2^2-\pi^2/L^2)^2}\,dk_1dk_2,\\
&\delta(\alpha)=E_2(\alpha)-E_1(\alpha).
\end{split}
\ee
Then one can obtain an asymptotic approximation of
$\delta(\alpha)=E_2(\alpha)-E_1(\alpha)$ when $0<2-\alpha\ll1$.
Extension to \eqref{eq:eigfso} with
 $n=3$, $\Omega=(0,1)\times(0,L_1)\times (0,L_2)$ with
$0<L_2\le L_1\le 1$ and $V(\bx)\equiv 0$ for $\bx\in \Omega$
can be done in a similar way. The details are omitted here
for brevity.

Unlike the case for the local FSO, for the FSO \eqref{eq:eigfso},
it is difficult to get a concise  lower bound of $\delta(\alpha)$
based on the asymptotic result \eqref{asym:box_1d} in 1D and \eqref{asym:box_2d1} in 2D.
Since our aim is not to get an optimal lower bound of $\delta(\alpha)$,
one idea is to check whether the lower bound for the local FSO obtained in the previous section remains valid for the FSO.
In order to do so, Fig \ref{fig:box_1d_gap} compares the fundamental gaps
of \eqref{eq:eigfso} obtained numerically, the asymptotic approximations
given in \eqref{asym:box_1d} for 1D, 
and the lower bounds of $\delta_{\rm loc}(\alpha)$
given in \eqref{gaplc234} (or \eqref{conj1}) and \eqref{lgap765} for $n=1$ 
and $0\le 2-\alpha \ll1$.

\begin{figure}[htbp]
\centerline{\psfig{figure=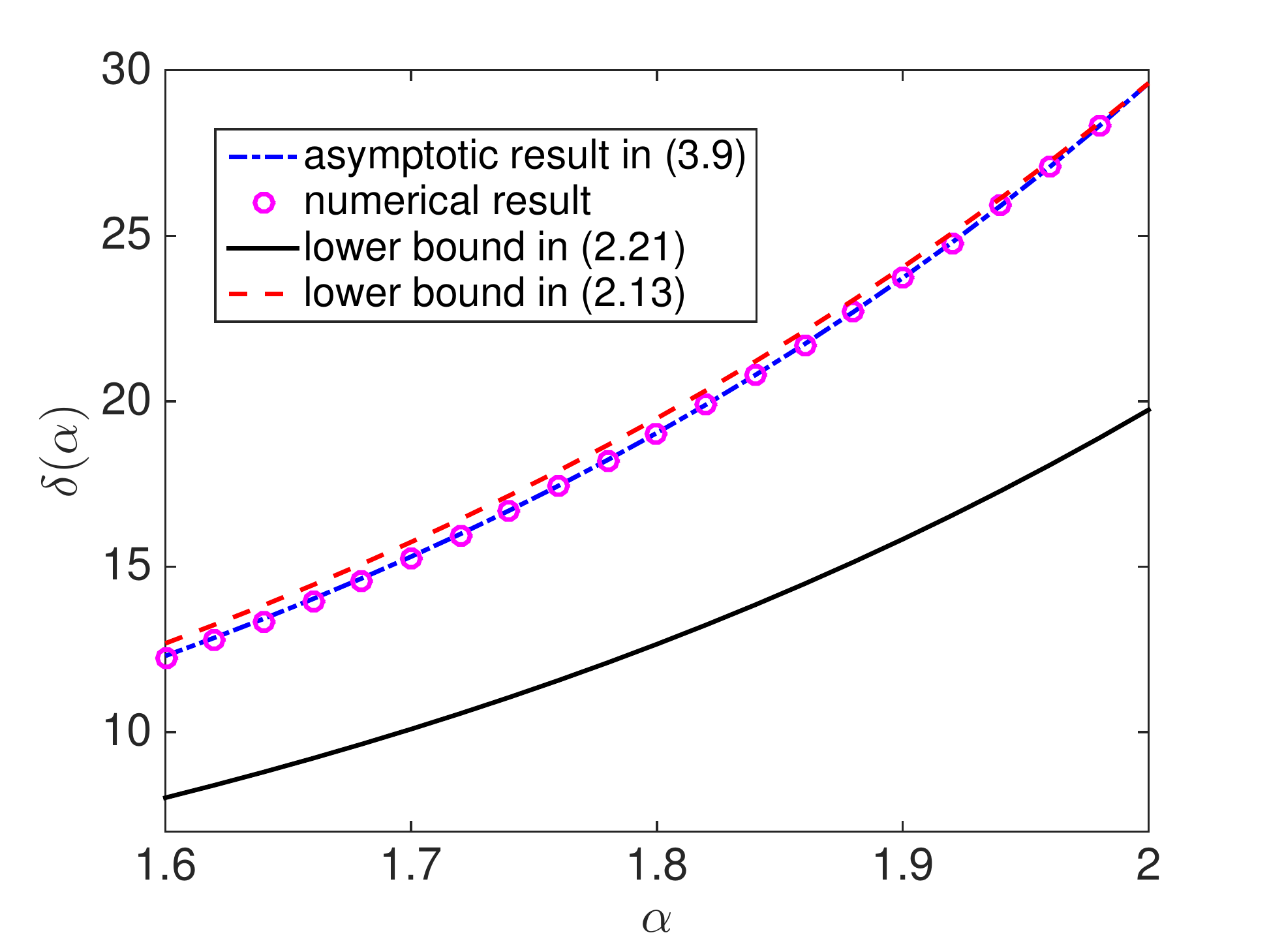,height=6cm,width=10cm,angle=0}}
\caption{Comparison of the lower bound in \eqref{lgap765} (solid line) and
in \eqref{gaplc234} (dotted line),
numerical results (dash line)  and asymptotic results in
\eqref{asym:box_1d} 
(dash-dot line) for the fundamental gap $\delta(\alpha)$ of the FSO
\eqref{eq:eigfso} with $n=1$,
$\Omega=(0,1)$ and $V_\bog(x)\equiv 0$ for $x\in \Omega$.
}
\label{fig:box_1d_gap}
\end{figure}

From Fig. \ref{fig:box_1d_gap}, we can see that: (i)
our asymptotic results agree with the numerical results very well
when $0\le 2-\alpha\ll 1$; (ii) the lower bound of $\delta_{\rm loc}(\alpha)$
given in \eqref{lgap765} is still a lower bound of $\delta(\alpha)$;
 and (iii) when $n=1$, the lower bound of $\delta_{\rm loc}(\alpha)$
given in \eqref{conj1} is not a lower bound of $\delta(\alpha)$.
With these observations, we will test numerically
whether the lower bound of $\delta_{\rm loc}(\alpha)$
given in \eqref{lgap765} is still a lower bound of $\delta(\alpha)$
for general geometry and general potential in the next subsection.

\subsection{Numerical results for general potentials}
Numerical solution of the eigenvalue problem \eqref{eq:eigfso} is very challenging due to the non-local boundary condition in an unbounded domain.
There exist some numerical methods for PDEs with fractional Laplacian in unbounded domains  based on finite-difference methods (cf. \cite{Duo,huang2014} and spectral methods (cf. \cite{klein2014,mao2017}). In \cite{Sheng182}, we developed a promising method using the mapped Chebyshev functions for solving PDEs with fractional Laplacian in unbounded domain.
 We adopt this method  to solve \eqref{eq:eigfso} numerically.
Thanks to the scaling property shown in Lemma \ref{lem:scalefso},
 the diameter of the domain
$\Omega$ is always  taken as $D=1$.

Fig. \ref{fig:box_1d_gap2} shows the numerical results
on the fundamental gap $\delta(\alpha)$ of \eqref{eq:eigfso}
when $n=1$, $\Omega=(0,1)$  with different external potentials $V_\bog(x)$.
Fig. \ref{fig:box_2d_numer} shows similar results
with $n=2$, different $\Omega$ and different
external potentials $V_\bog(x,y)$.

\begin{figure}[htbp]
\centerline{\psfig{figure=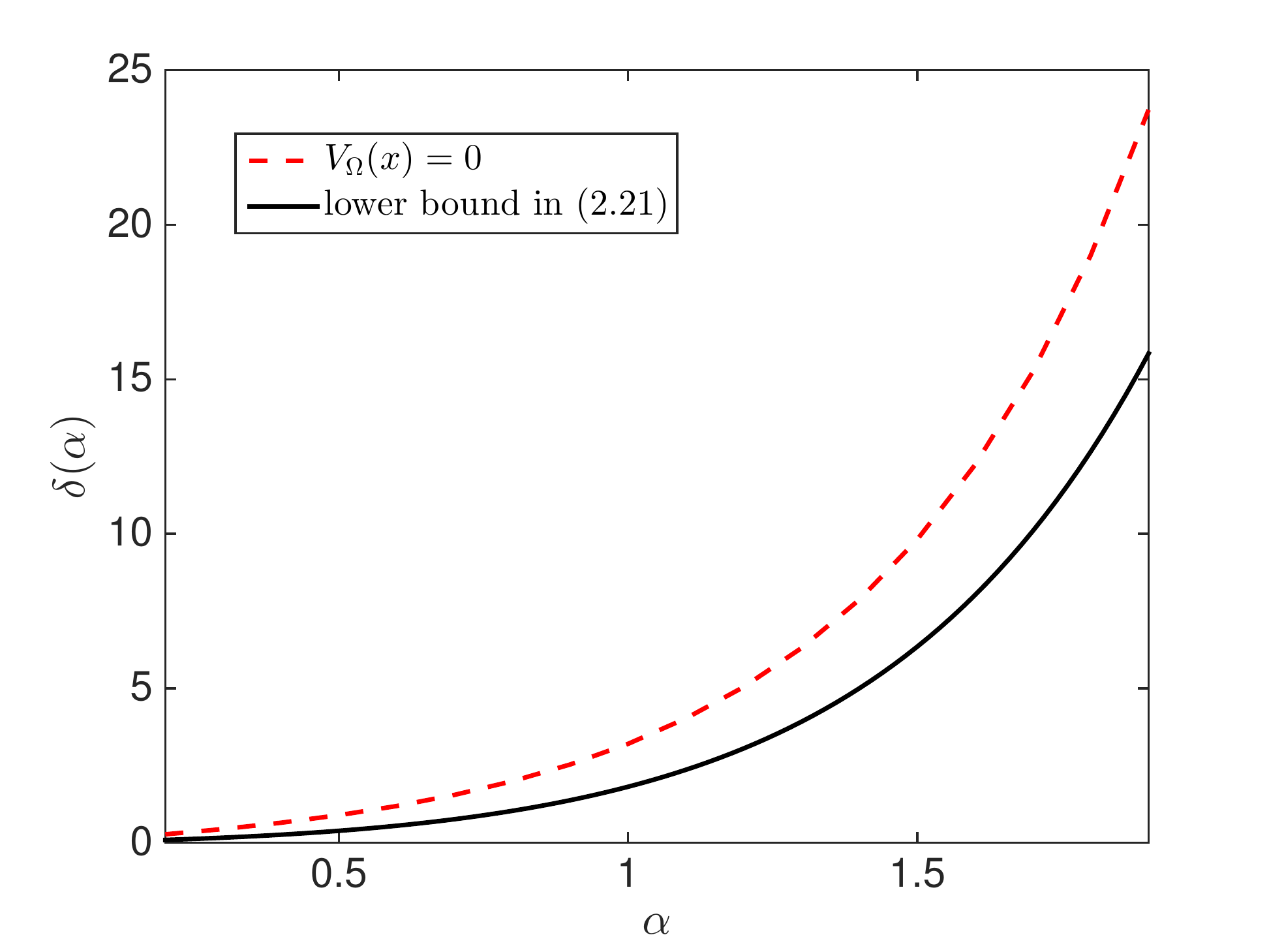,height=6cm,width=7cm,angle=0}
\psfig{figure=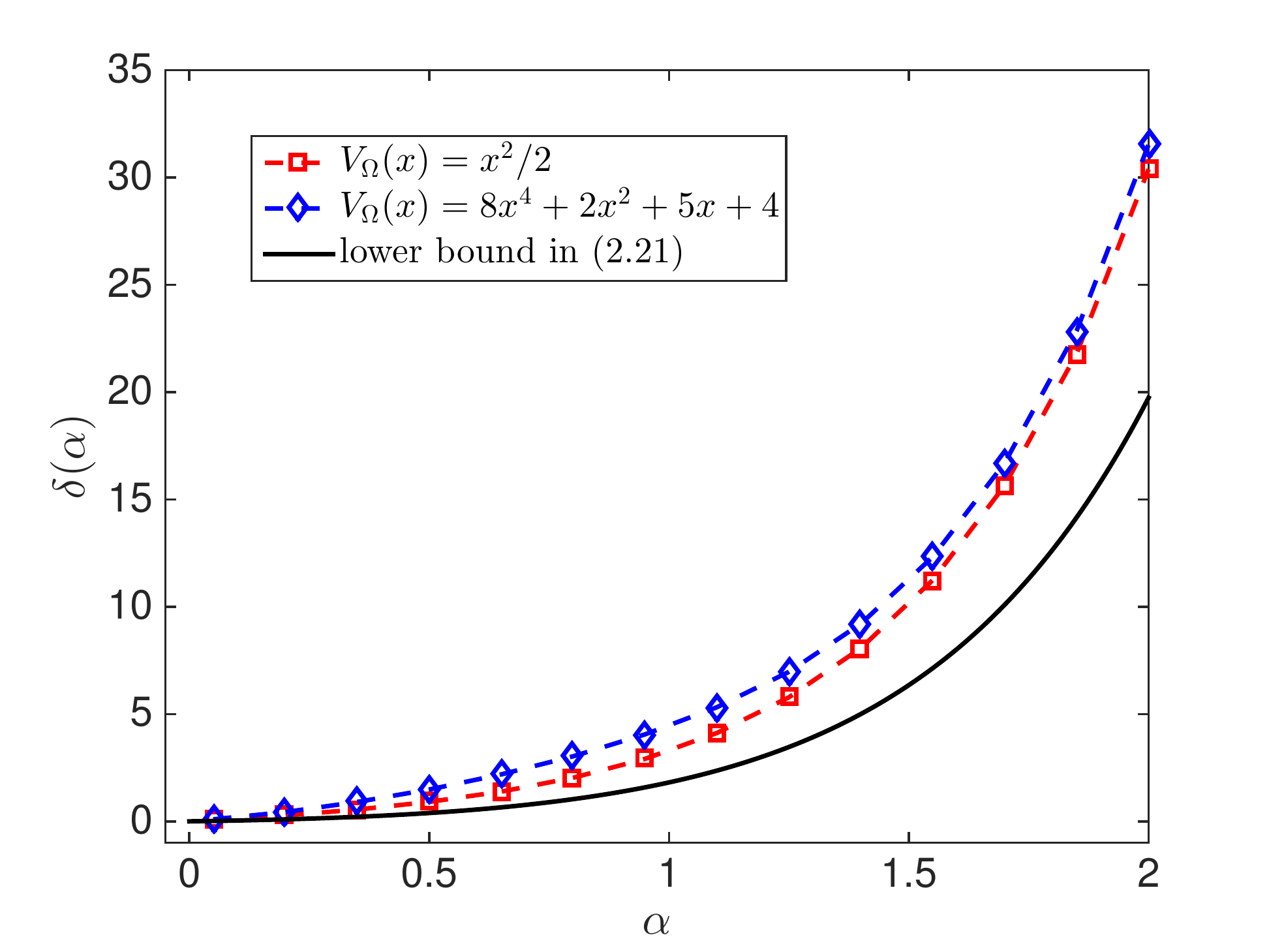,height=6cm,width=7cm,angle=0}}
\caption{Comparison of the lower bound in \eqref{lgap765} (solid line) and
numerical results (dash lines)  for the fundamental gap $\delta(\alpha)$ of the FSO \eqref{eq:eigfso} with  $n=1$, $\Omega=(0,1)$ and
$V_\bog(x)\equiv0$ (left) or different convex potentials $V_\bog(x)$ (right).}
\label{fig:box_1d_gap2}
\end{figure}

\begin{figure}[htbp]
\centerline{\psfig{figure=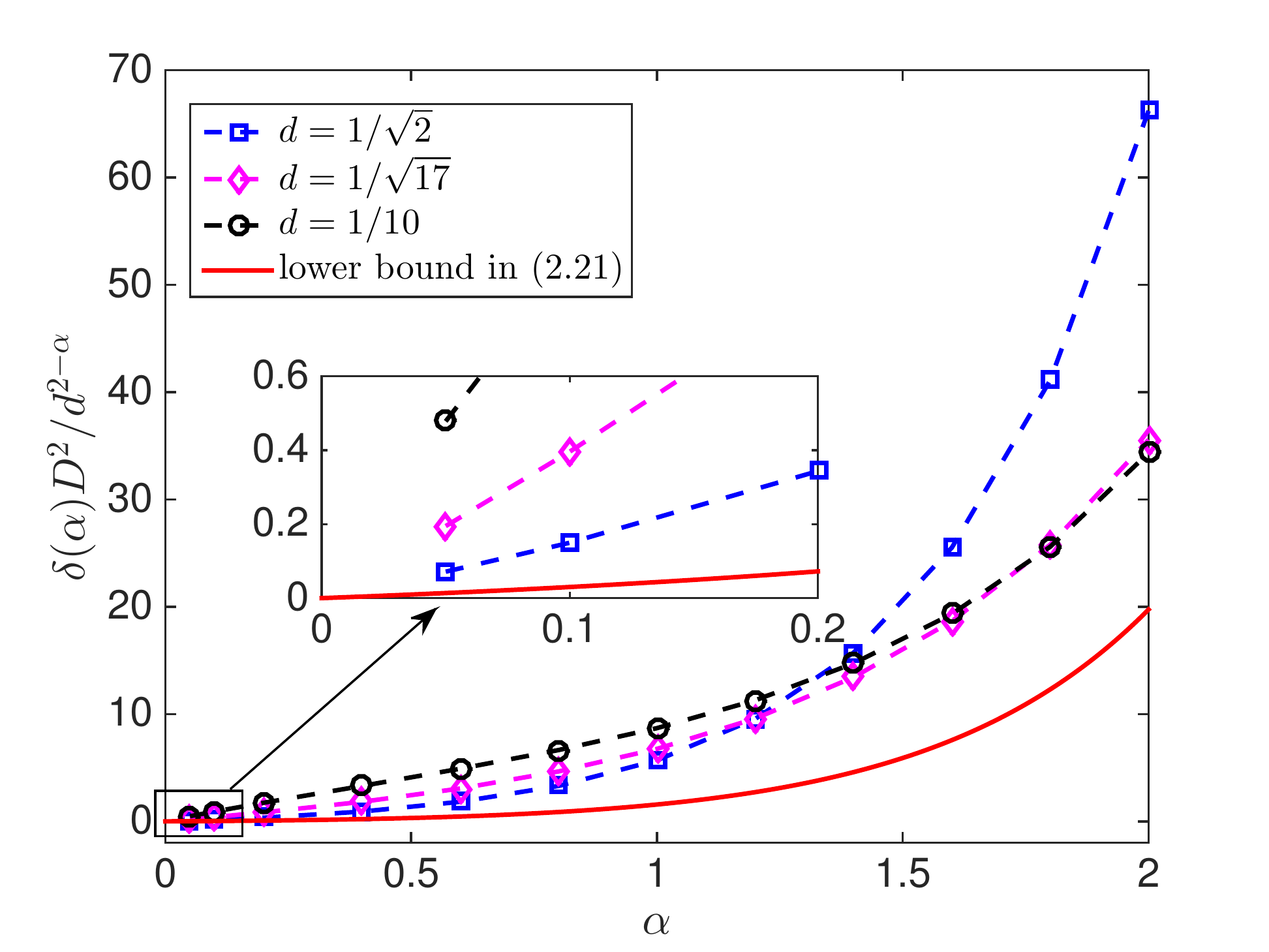,height=6cm,width=7cm,angle=0}
\psfig{figure=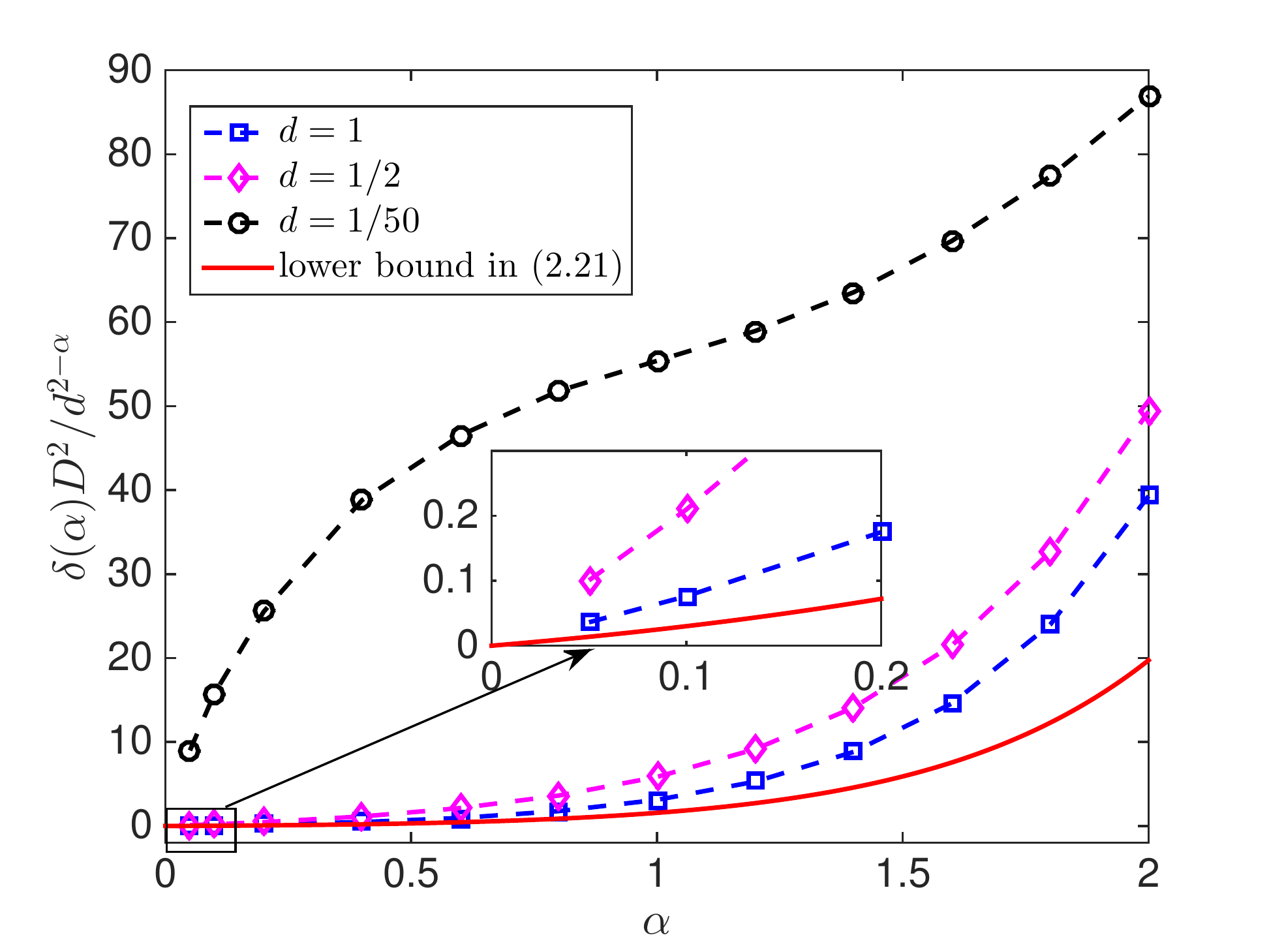,height=6cm,width=7cm,angle=0}}
\centerline{\psfig{figure=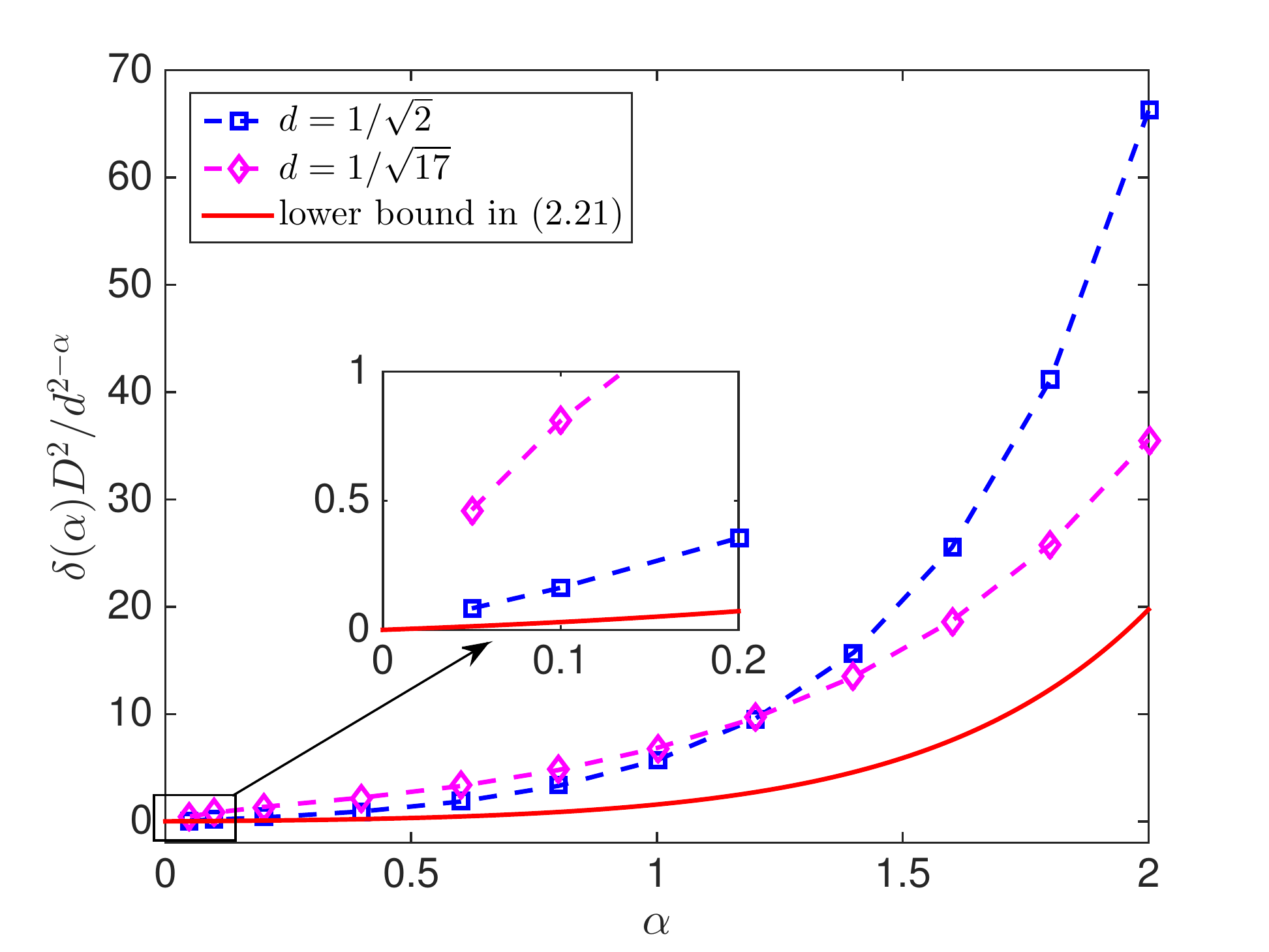,height=6cm,width=7cm,angle=0}
\psfig{figure=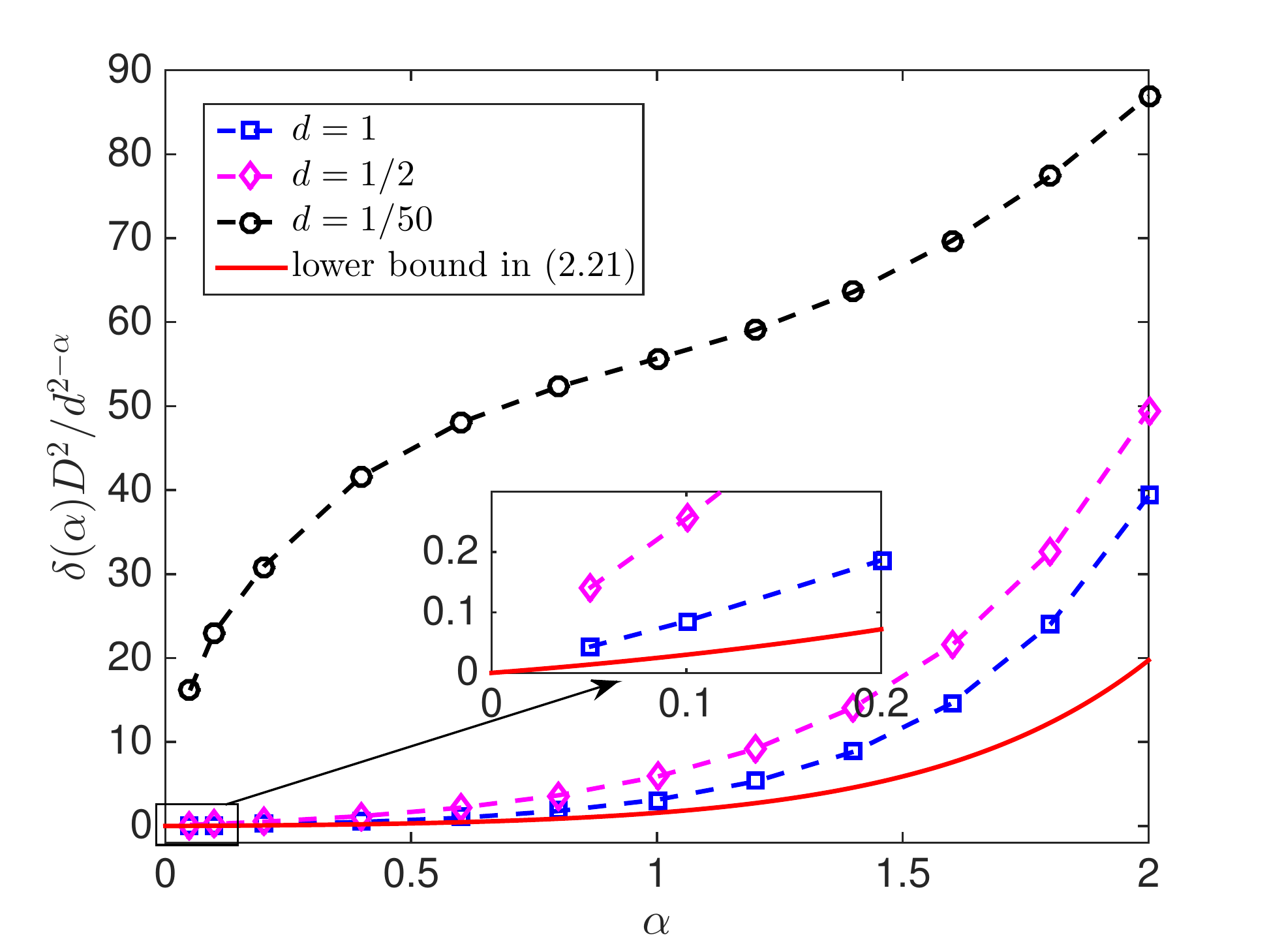,height=6cm,width=7cm,angle=0}}
\caption{Comparison of the lower bound in \eqref{lgap765} (solid line) and
numerical results (dash lines)  for the fundamental gap $\delta(\alpha)$ of the FSO \eqref{eq:eigfso} with  $n=2$ and a rectangular type domain $\Omega=(0,d)\times(0,\sqrt{1-d^2})$ with different $0<d< 1$(left);
and an elliptic domain $\Omega=\{(x,y)\,|\,x^2+y^2/d^2\le1\}$ with different
$0<d\le 1$ (right). The external potential is chosen as $V_\bog(x,y)\equiv0$ (top) or $V_\bog(x,y)=(x^2+y^2)/2$ (bottom) in $\Omega$.}
\label{fig:box_2d_numer}
\end{figure}

Again, based on our asymptotic results in the previous subsection and numerical results in  Figs. \ref{fig:box_1d_gap2}\&\ref{fig:box_2d_numer}
as well as extensive  numerical results which draw similar conclusion and thus are not shown here for brevity,
we are confident to formulate the gap conjecture
\eqref{conj2} for the FSO \eqref{eq:eigfso}.

\section{The fundamental gaps of the FSO \eqref{eq:fso} in the whole space}\label{sec:global_classic}
In this section, we will study asymptotically and numerically the first two smallest eigenvalues and their corresponding eigenfunctions of the eigenvalue problem \eqref{eq:eig} generated by the FSO \eqref{eq:fso} in the whole space
and then formulate a gap conjecture. Here we assume $V(\bx)\in L^\infty_{\rm loc}({\mathbb R}^n)$.

  In many applications \cite{BC}, the following harmonic potential is widely used
\be \label{harmold}
V(\bx)=\sum_{j=1}^n \gamma_j^2 x_j^2, \qquad
\bx=(x_1,\ldots,x_n)^T\in {\mathbb R}^n,
\ee
where $\gamma_1>0$, $\ldots$, $\gamma_n>0$ are given positive constants.
Without loss of generality, we assume that $0<\gamma_1\le \ldots \le \gamma_n$. Denote $\gm:=\gm_1$ and $\eta_j:=\frac{\gm_j}{\gm_1}\ge1$ ($j=1,\ldots,n$) and $\eta =\max_{1\le j\le n} \eta_j=\eta_n=\frac{\gm_n}{\gm_1}\ge1$, then the harmonic potential \eqref{harm}
can be re-written as
\be \label{harm}
V(\bx)=\gm^2 \sum_{j=1}^n \eta_j^2 x_j^2=\gm^2\left(x_1^2+\sum_{j=2}^n \eta_j^2 x_j^2\right), \qquad
\bx=(x_1,\ldots,x_n)^T\in {\mathbb R}^n.
\ee

\subsection{Scaling property}
Introduce
\begin{equation}\label{scaxpw}
D:=\gm^{-\frac{2}{2+\alpha}}, \quad \tbx=\frac{\bx}{D},
\quad \tilde{V}(\tbx) =\tilde x_1^2+\sum_{j=2}^n \eta_j^2 \tilde x_j^2, \quad
\tilde \phi(\tbx)=\phi(\bx), \quad \bx,\tbx\in{\mathbb R}^n,
\end{equation}
and consider the re-scaled eigenvalue problem
\begin{equation}\label{eq:eiglrw}
\tilde L_{\rm FSO}\,\tilde\phi(\tbx):=\left[(-\Delta)^{\frac{\alpha}{2}}+\tilde {V}(\tbx)\right]\tilde\phi(\tbx)=
\tilde E\,\tilde\phi(\tbx),\qquad \tbx\in {\mathbb R}^n,
\end{equation}
then we have

\smallskip

\begin{lemma}\label{lem:scalew}
Let $E$ be an eigenvalue of \eqref{eq:eig} with \eqref{harm} and $\phi:=\phi(\bx)$ is the corresponding eigenfunction, then $\tilde E =\gm^{-\frac{2\alpha}{2+\alpha}} E$
is an eigenvalue of \eqref{eq:eiglrw} and  $\tilde \phi:=\tilde \phi(\tbx) = \phi(D\tbx)=\phi(\bx)$ is the corresponding eigenfunction, which immediately imply the scaling property on the fundamental gap $\delta(\alpha)$ of
\eqref{eq:eig} with \eqref{harm} as
\be
\delta(\alpha) = \gm^{\frac{2\alpha}{2+\alpha}}\; \tdelta(\alpha), \qquad
0<\alpha\le 2,
\ee
where $\tdelta(\alpha)$ is the fundamental gap of \eqref{eq:eiglrw}.
\end{lemma}

\smallskip

\begin{proof}
From \eqref{scaxpw}, similar to \eqref{dtapsc}, we have
\be\label{dtapscw}
(-\Delta)^{\alpha/2}\phi(\bx)
=D^{-\alpha}\,(-\Delta)^{\alpha/2}\tilde\phi(\tbx), \quad \bx,\tbx\in {\mathbb R}^n.
\ee
Substituting \eqref{dtapscw} into \eqref{eq:eig} with \eqref{harm}, noting \eqref{harm}-\eqref{eq:eiglrw}, we get
\bea
E\,\tphi(\tbx)&=&E\,\phi(\bx)=\left[(-\Delta)^{\frac{\alpha}{2}}+V(\bx)\right]
\phi(\bx)=\left[D^{-\alpha}\,(-\Delta)^{\frac{\alpha}{2}} +V(D\tbx)\right]\tphi(\tbx)\nonumber\\
&=&D^{-\alpha}
\left[(-\Delta)^{\frac{\alpha}{2}} +D^{2+\alpha}\gm^2 \tilde V(\tbx)\right]
\tphi(\tbx)
=\gm^{\frac{2\alpha}{2+\alpha}}
\left[ (-\Delta)^{\frac{\alpha}{2}}+\tilde V(\tbx)\right]
\tilde \phi(\tbx),
\eea
where $\bx,\tbx\in {\mathbb R}^n$ and $D=\gm^{-2/(2+\alpha)}$,
which immediately implies that $\tilde \phi(\tbx)$ is an eigenfunction of the operator $(-\Delta)^{\frac{\alpha}{2}}+\tilde V(\tbx)$ with the eigenvalue
$\tilde E=\gm^{-\frac{2\alpha}{2+\alpha}} E$.
\end{proof}

\subsection{Asymptotic results for harmonic potential when $0\le 2-\alpha\ll1$}
Consider a harmonic potential in \eqref{eq:eig} as \eqref{harmold} (or \eqref{harm}). By using the Fourier transform over ${\mathbb R}^n$,
the eigenvalue problem \eqref{eq:eig} can
be reformulated as a standard eigenvalue problem in the phase (or Fourier)
space as, i.e. without the fractional Laplacian operator
\begin{equation}\label{eq:eigFsp}
\left(-\sum_{j=1}^n \gamma_j^2 \frac{\partial^2}{\partial k_j^2} +|\bk|^\alpha\right)\hat\phi(\bk)=E\,\hat\phi(\bk),
\qquad \bk=(k_1,\ldots,k_n)^T\in \mathbb{R}^n,
\end{equation}
where $\hat{\phi}(\bk)$ is the Fourier transform of $\phi(\bx)$ over
the whole space $\mathbb{R}^n$. Introduce
\be
\tilde k_j=\frac{k_j}{\gm_j}, \quad j=1,\ldots,n,\quad
\tilde \phi(\tilde \bk)=\phi(\bk)=\phi(\gm_1\tilde k_1,\ldots,
\gm_n\tilde k_n),  \quad \bk,\tilde \bk\in{\mathbb R}^n,
\ee
then the eigenvalue problem \eqref{eq:eigFsp} can be reformulated as
an eigenvalue with the Laplacian
\begin{equation}\label{eq:eigFspn}
\left(-\Delta +\left(\sum_{j=1}^n \gm_j^2|\tilde k_j|^2\right)^{\alpha/2}\right)\tilde\phi(\tilde \bk)=E\,\tilde\phi(\tilde \bk),
\qquad \tilde\bk=(\tilde k_1,\ldots,\tilde k_n)^T\in \mathbb{R}^n,
\end{equation}
In fact, if $\phi(\bx)\ne 0$ is an eigenfunction of \eqref{eq:eig} corresponding to the eigenvalue $E$, then
$\hat\phi(\bk)\ne0$ is an eigenfunction of \eqref{eq:eigFsp} corresponding
to the same eigenvalue $E$, and $\tilde\phi(\tilde \bk)\ne0$ is an eigenfunction of \eqref{eq:eigFspn} corresponding
to the same eigenvalue $E$. In addition, we have
\bea\label{def:E_har}
E&=&\frac{1}{\int_{\mathbb{R}^n}|\hat{\phi}(\bk)|^2\,d\bk}
\int_{\mathbb{R}^n}\left(\sum_{j=1}^n \gamma_j^2 \left|\frac{\partial\hat\phi(\bk)}{\partial k_j}\right|^2 +|\bk|^\alpha\left|\hat\phi(\bk)\right|^2\right)\,d\bk\nonumber\\
&=&\frac{1}{\int_{\mathbb{R}^n}|\tilde{\phi}(\tilde \bk)|^2\,d\tilde\bk}
\int_{\mathbb{R}^n}\left(|\nabla \tilde \phi(\tilde\bk)|^2+\left(\sum_{j=1}^n \gamma_j^2 |\tilde k_j|^2\right)^{\alpha/2}\left|\tilde\phi(\tilde\bk)\right|^2\right)\,d\tilde\bk.
\eea

\begin{lemma}\label{har1dwhol}
Taken $n=1$ and a harmonic potential $V(x)$  as \eqref{harm}
in \eqref{eq:eig}, when $0\le \varepsilon:=2-\alpha\ll1$, we
have
\be\label{asym:har}
\delta(\alpha)\approx \gm+\frac{\alpha\gm^{\alpha/2}}{\sqrt{\pi}}
\Gamma\left(\frac{1+\alpha}{2}\right).
\ee
\end{lemma}

\smallskip

\begin{proof}
When $n=1$ and $\alpha=2$, the first two smallest eigenvalues and their
corresponding eigenfunctions of the eigenvalue problem \eqref{eq:eig}
with \eqref{harm} can be given as \cite{BC,BL}
\be\label{egwhol987}
\begin{split}
&E_1(2)=\gm, \qquad \phi_1(x)=\left(\frac{\gm}{\pi}\right)^{1/4}e^{-\gm x^2/2}, \qquad
x\in{\mathbb R},\\
&E_2(2)=3\gamma, \qquad \phi_2(x)=\sqrt{2\gm}x\left(\frac{\gm}{\pi}\right)^{1/4}e^{-\gm x^2/2}.\\
\end{split}
\ee
The Fourier transform of $\phi_l(x)$ ($l=1,2$) can be computed as
\be\label{philf98}
\hat{\phi}_1(k)=\frac{\sqrt{2}\pi^{1/4}}{\gm^{1/4}}e^{-\frac{k^2}{2\gm}},
\quad
\hat{\phi}_2(k)=\frac{-2i\pi^{1/4}}{\gm^{3/4}}ke^{-\frac{k^2}{2\gm}},
\qquad k\in{\mathbb R}.
\ee
When $\alpha$ satisfies $0\le 2-\alpha\ll 1$,
the two normalized eigenfunctions $\phi_l^{(\alpha)}(x)$ ($l=1,2$) corresponding to the first two smallest eigenvalues
of \eqref{eq:eig} can be well approximated by $\phi_l(x)$ ($l=1,2$), respectively, i.e.
\be\label{efbda12w}
\phi_1^{(\alpha)}(x)\approx \phi_1(x),
\qquad \phi_2^{(\alpha)}(x)\approx \phi_2(x), \qquad
x\in {\mathbb R}.
\ee
Substituting \eqref{efbda12w} and \eqref{philf98} into \eqref{def:E_har},
we can obtain
the approximations of the first two smallest eigenvalues $E_l(\alpha)$
($l=1,2$) as
\be\label{E1E2wol}
\begin{split}
&E_1{(\alpha)}=E^{(\alpha)}(\phi_1^{(\alpha)})\approx E^{(\alpha)}(\phi_1)= \frac{\gm}{2}+\frac{\gm^{\alpha/2}}{\sqrt{\pi}}\Gamma
\left(\frac{1+\alpha}{2}\right),\\
&E_2{(\alpha)}=E^{(\alpha)}(\phi_2^{(\alpha)})\approx E^{(\alpha)}(\phi_2)=  \frac{3\gm}{2}+\frac{2\gm^{\alpha/2}}{\sqrt{\pi}}
\Gamma\left(\frac{3+\alpha}{2}\right).
\end{split}
\ee
Subtracting the first equation from the second equation in \eqref{E1E2wol}, we get
\bea
\delta(\alpha)&=&E_2{(\alpha)}-E_1{(\alpha)}\approx
\gm+\frac{2\gm^{\alpha/2}}{\sqrt{\pi}}
\Gamma\left(1+\frac{1+\alpha}{2}\right)-\frac{\gm^{\alpha/2}}{\sqrt{\pi}}\Gamma
\left(\frac{1+\alpha}{2}\right)\nonumber\\
&=&\gm+\frac{(1+\alpha)\gm^{\alpha/2}}{\sqrt{\pi}}
\Gamma\left(\frac{1+\alpha}{2}\right)-\frac{\gm^{\alpha/2}}{\sqrt{\pi}}\Gamma
\left(\frac{1+\alpha}{2}\right)=\gm+\frac{\alpha\gm^{\alpha/2}}{\sqrt{\pi}}
\Gamma\left(\frac{1+\alpha}{2}\right).
\eea
The proof is completed.
\end{proof}

Similarly, taken $n=2$ and $1=\gamma\le \eta_2=\eta$ in
\eqref{eq:eig} and \eqref{harm}, when $0\le \varepsilon:=2-\alpha\ll1$,
we get (with details omitted here for brevity)
\begin{align}
\delta(\alpha)&\approx 1-\frac{1}{\pi\sqrt{\eta}}\iint\, (2k_1^2+1)(k_1^2+k_2^2)^{\alpha}e^{-(k_1^2+k_2^2/\eta)}\,dk_1dk_2\nonumber\\
&=1-\frac{\Gamma(-(1+\alpha)/2)\Gamma(1+\alpha/2)}{\sqrt{\pi\eta}\Gamma(-\alpha/2)}\left[-_2F_1(1/2,1+\alpha/2;(3+\alpha)/2;1/\eta)\right.\nonumber\\
&\left.+(2+\alpha)\,\,\,_2F_1(1/2,2+\alpha/2;(3+\alpha)/2;1/\eta)\right]\nonumber\\
&+\frac{\alpha\sqrt{\pi}\eta^{\alpha/2}}{\eta-1}
\left[-\eta\,\,\,_2F_1(-1/2,-\alpha/2;(1-\alpha)/2;1/\eta)
/\Gamma((1-\alpha)/2)\right.\nonumber\\
&\left.+(\eta-1)\,\,\, _2F_1(1/2,-\alpha/2;(1-\alpha)/2;1/\eta)/\Gamma((1-\alpha)/2)
\sec(\alpha\pi)\right].\label{gap:har_2D}
\end{align}

In order to verify the asymptotic results \eqref{asym:har} in 1D and
\eqref{gap:har_2D} in 2D when $0\le 2-\alpha\ll1$, Fig. \ref{fig:gap_har_asym} plots the asymptotic results and numerical results
of the fundamental gap $\delta(\alpha)$ of the FSO \eqref{eq:eig} when
$0\le 2-\alpha\ll 1$. The results indicate
that our asymptotic results are quite accurate in the regime  $0\le 2-\alpha\ll1$ (cf. Fig. \ref{fig:gap_har_asym}). In addition,
we cannot get a lower bound of the fundamental gap
$\delta(\alpha)$ from the asymptotic results!

\begin{figure}[htbp]
\centerline{\psfig{figure=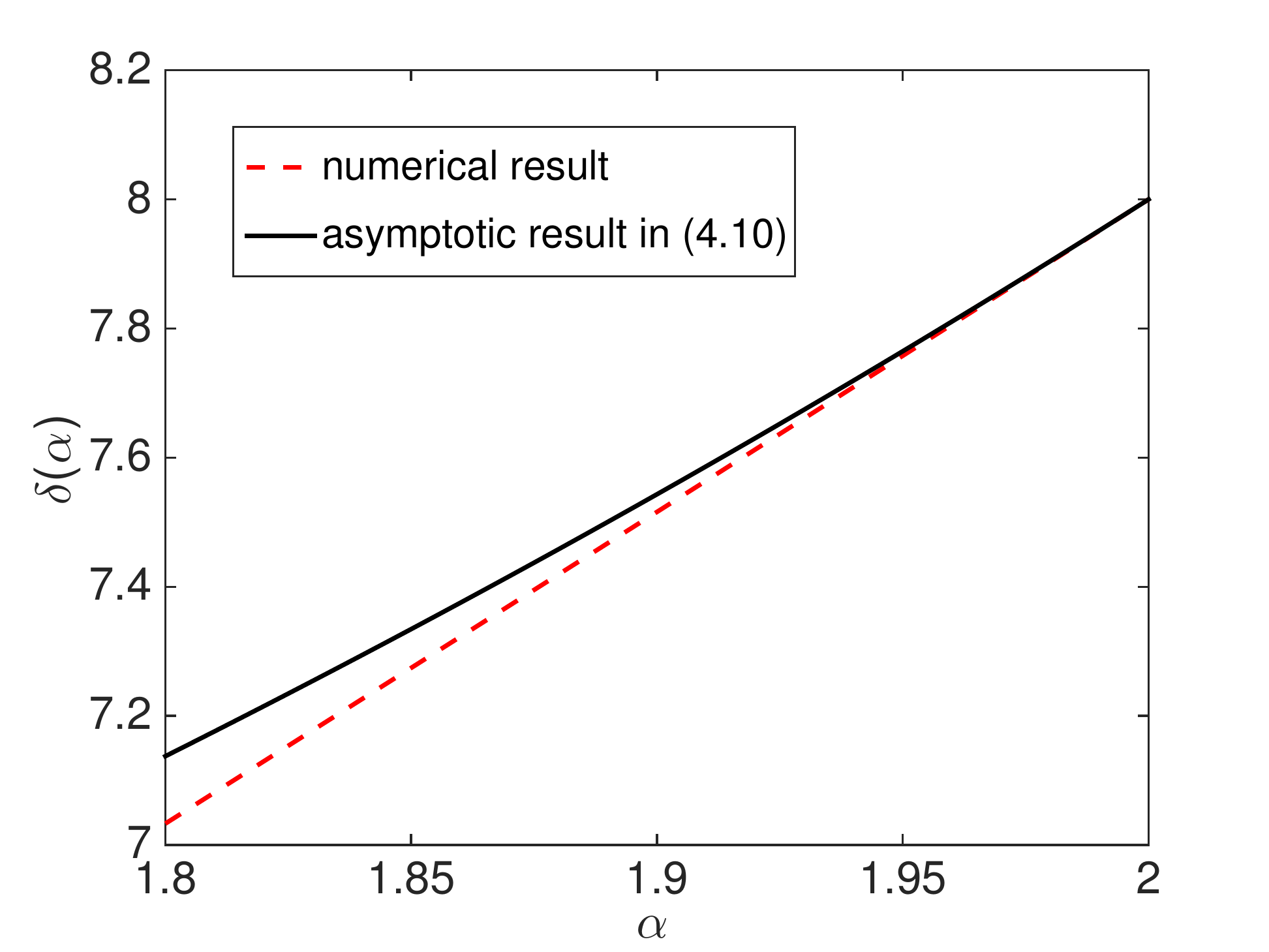,height=6cm,width=7cm,angle=0}
\psfig{figure=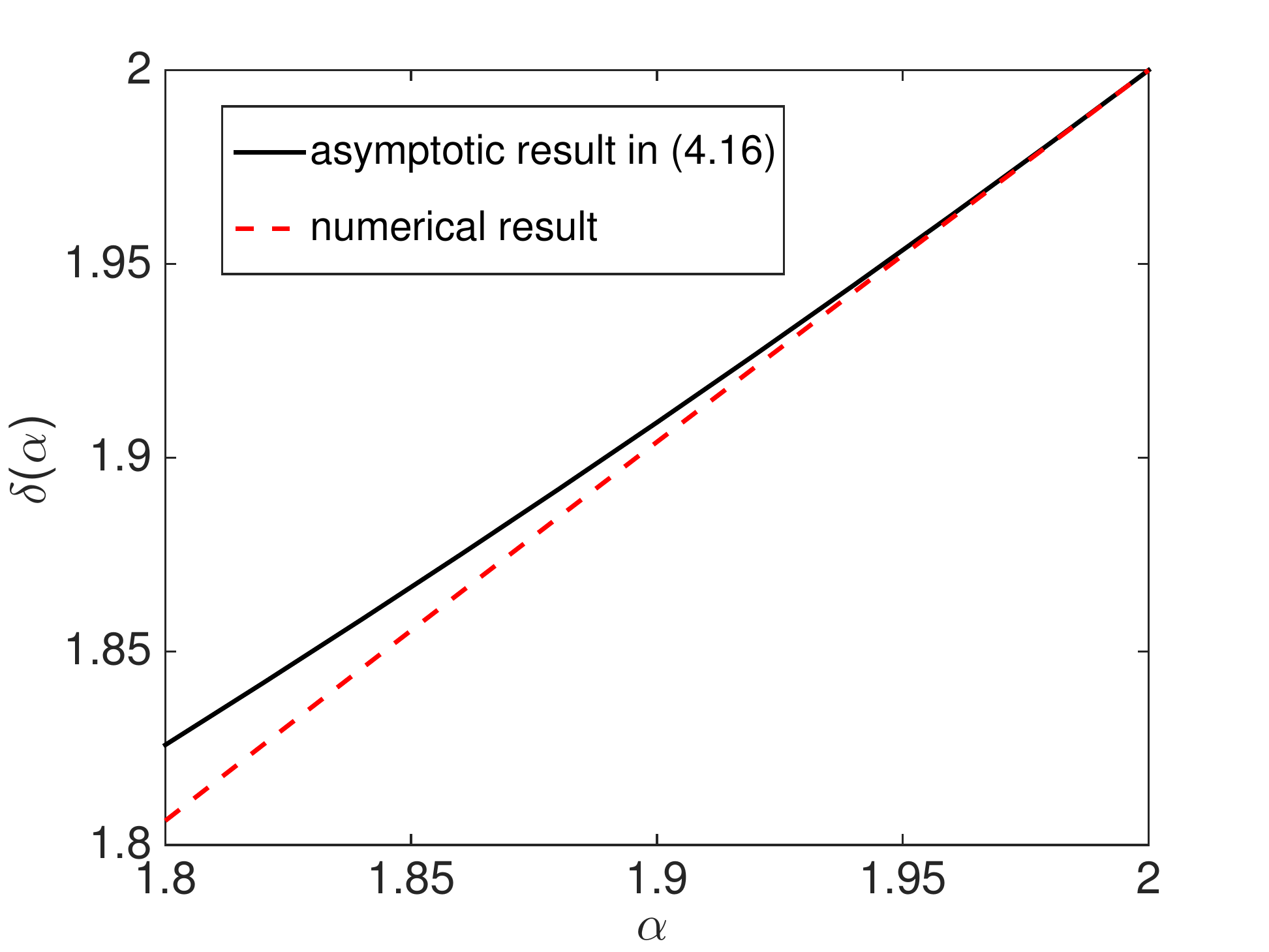,height=6cm,width=7cm,angle=0}}
\caption{Comparison of the asymptotic results in \eqref{asym:har}
or \eqref{gap:har_2D} (solid line) and
numerical results (dash lines)
for the fundamental gap $\delta(\alpha)$ of
the FSO  \eqref{eq:eig} with $n=1$ and
$V(x)=16x^2$ (left) and with $n=2$ and $V(x,y)=x^2+16y^2$ (right).}
\label{fig:gap_har_asym}
\end{figure}

\subsection{A formal lower bound on the fundamental gap in 2D}
In order to get a lower bound of the fundamental gap $\delta(\alpha)$
of the FSO \eqref{eq:eig}, we take $n=2$ and
$V(x,y)=x^2+\eta^2 y^2$ with $\eta\ge1$ in \eqref{eq:eig} and
consider the following eigenvalue problem
\be\label{2dex1}
\left[(-\Delta)^{\frac{\alpha}{2}}+(x^2+\eta^2y^2)\right]\phi(\bx)= E\,\phi(\bx), \qquad \bx=(x,y)^T\in{\mathbb R}^2.
\ee
When $\alpha=2$, the first two smallest eigenvalues of
\eqref{2dex1} are \cite{BC,BL}
\be\label{eig2dex1}
E_1:=E_1(2)=1+\eta, \qquad E_2:=E_2(2)=3+\eta, \qquad \eta\ge1.
\ee

Motivated by the methods and results in the previous two sections,
we assume that the lower bound of the fundamental gap
might depend on the parameter $\eta$ -- the anisotropy
of the harmonic potential.
Similar to the case of the local FSO,
i.e. finding the lower bound of the fundamental gap by estimating
$\lmd_2^{\alpha/2}-\lmd_1^{\alpha/2}$ with $\lmd_1$ and $\lmd_2$ being
the first two smallest eigenvalues of the corresponding
operator when $\alpha=2$, we formally assume that
the fundamental gap $\tilde \delta(\alpha)$ of \eqref{2dex1}
has a similar estimate as
\be\label{proof:har}
\tilde \delta(\alpha)\ge E_2^\beta -E_1^\beta=(3+\eta)^{\beta}-(1+\eta)^{\beta},
\ee
where $0<\beta\le 1$ is
to be determined in an asymptotic way by considering $\eta\to+\infty$.
When $\eta\gg1$, the eigenfunction of \eqref{2dex1}
varies extremely slow in the $x$-direction.
As a result, the problem \eqref{2dex1} can be formally
well approximated by
\be
\left[(-\partial_{yy})^{\frac{\alpha}{2}}+\eta^2y^2\right]u(y)=E\, u(y),
\qquad y\in{\mathbb R},
\ee
The scaling property in Lemma \ref{lem:scalew} implies that $E\sim \mathcal{O}(\eta^{2\alpha/(2+\alpha)})$, which indicates that  one reasonable choice of $\beta$ is
\be\label{beta865}
\beta=\frac{2\alpha}{2+\alpha}.
\ee

\begin{figure}[htbp]
\centerline{\psfig{figure=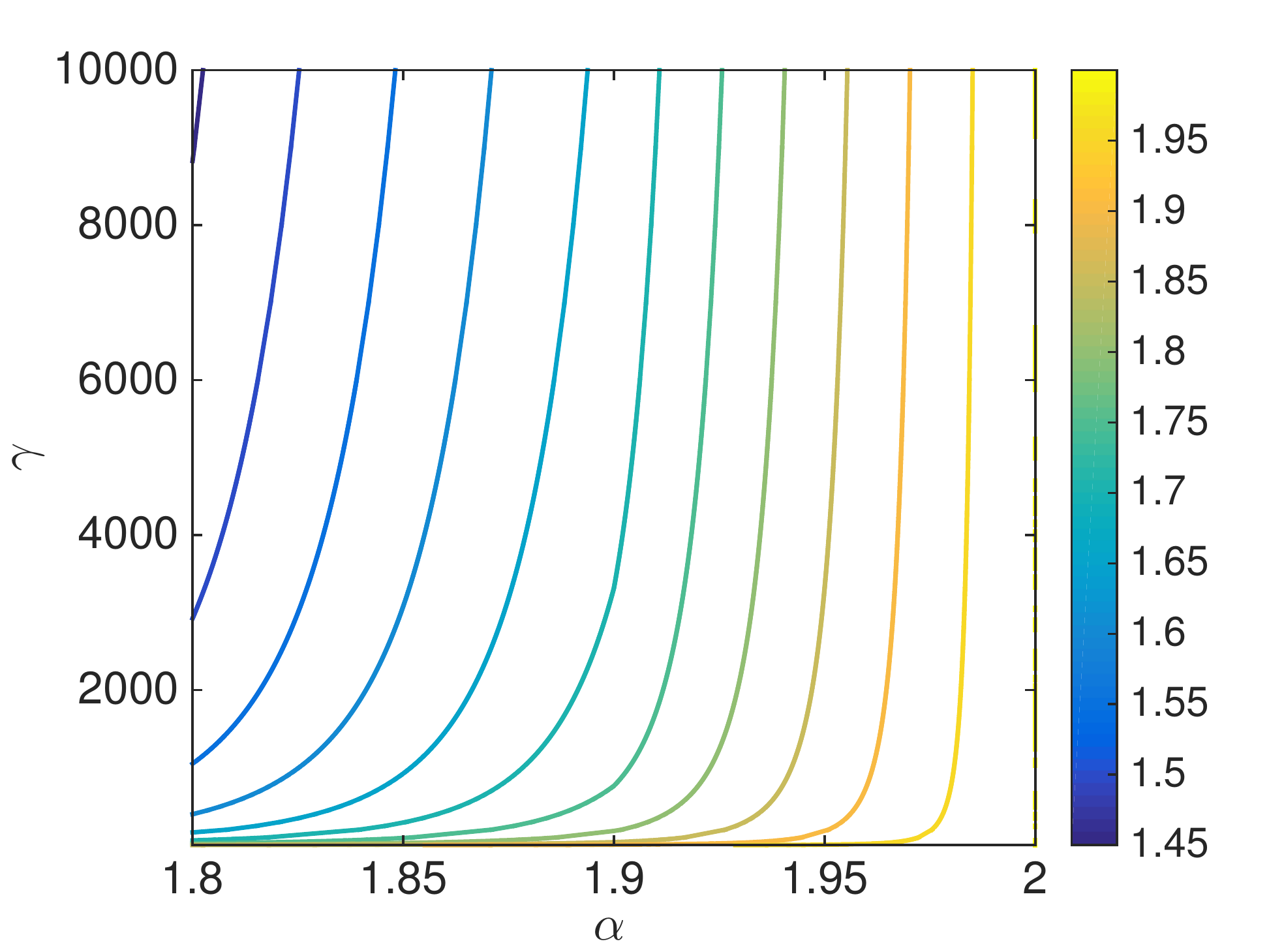,height=6cm,width=7cm,angle=0}
\psfig{figure=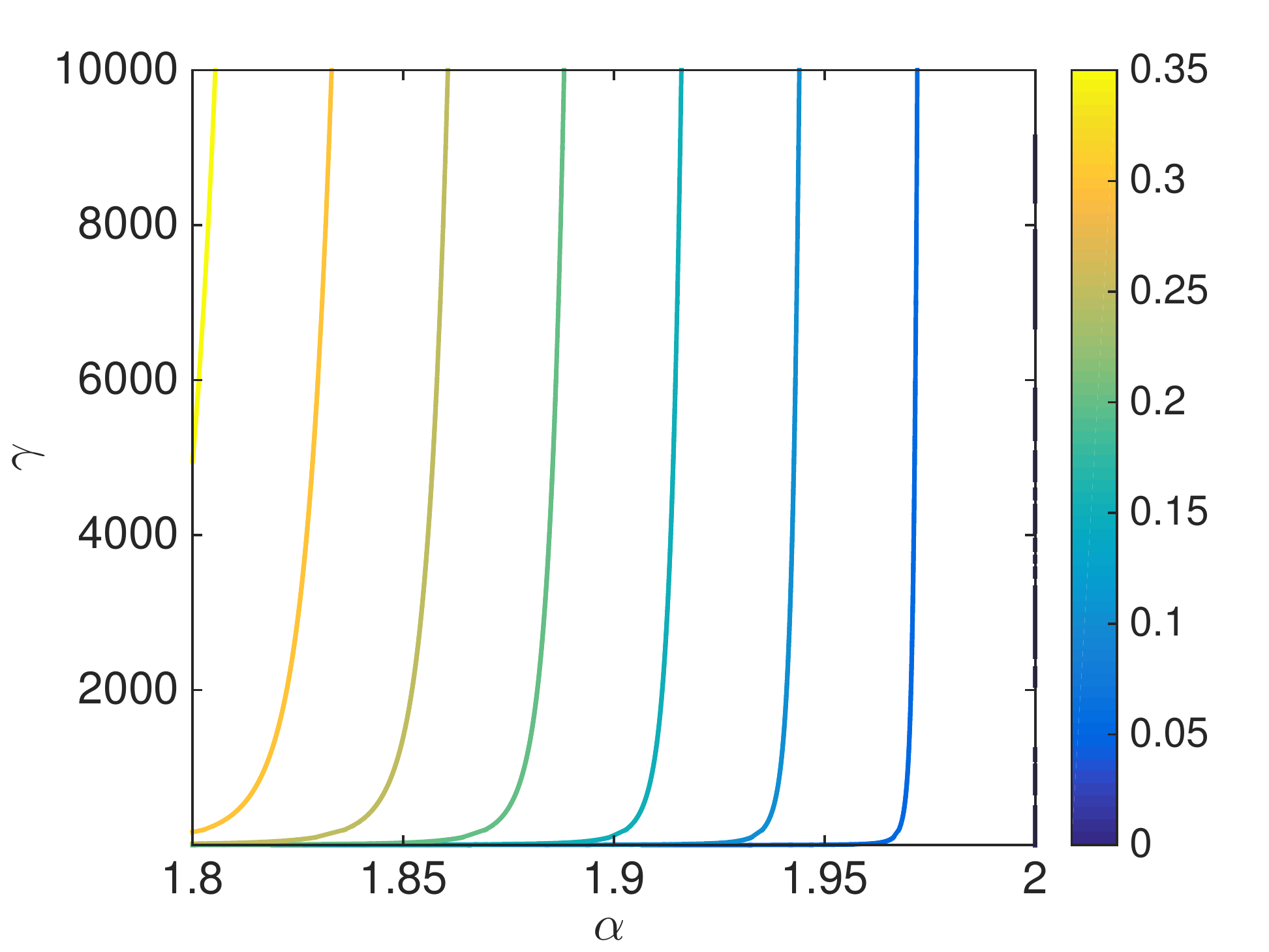,height=6cm,width=7cm,angle=0}}
\caption{Comparison of the asymptotic results in
\eqref{gap:har_2D} (left) and the lower bound in \eqref{lbdtl} (right)
for the fundamental gap $\tilde \delta(\alpha)$ of
the FSO  \eqref{2dex1} for different $\alpha$ and $\eta\ge1$.}
\label{fig:gap_har_2D_asym_first}
\end{figure}

When $\eta\ge1$, we have
\[
\tilde \delta(\alpha)\ge\eta^\beta
\left[\left(1+\frac{3}{\eta}\right)^\beta-
\left(1+\frac{1}{\eta}\right)^\beta\right]
=\eta^\beta \, \beta \, \frac{1}{(1+\xi)^{1-\beta}}\, \frac{2}{\eta}
= \frac{2\beta \eta^{\beta-1}}{(1+\xi)^{1-\beta}},
\]
where $\xi\in [1/\eta,3/\eta]\subset (0,3]$. Noting that $\frac{1}{(1+\xi)^{1-\beta}}$ is a decreasing function when $\xi\ge0$ and
taking $\xi=3$, we get
\be \label{dt985}
\tilde \delta(\alpha)\ge2\beta \eta^{\beta-1}\frac{1}{4^{1-\beta}}=
\frac{4^\beta \beta }{2}\eta^{\beta-1}.
\ee
Plugging \eqref{beta865} into \eqref{dt985}, we obtain a
lower bound
\be\label{lbdtl}
\tilde \delta(\alpha)\ge2^{\frac{4\alpha}{2+\alpha}}\,
\frac{ \alpha }{2+\alpha}\;
\frac{1}{\eta^{\frac{2-\alpha}{2+\alpha}}}.
\ee

To compare the asymptotic results
\eqref{gap:har_2D} in 2D and the formal lower bound
in \eqref{lbdtl} for the fundamental gap $\tilde \delta(\alpha)$
of the FSO \eqref{2dex1},  Fig. \ref{fig:gap_har_2D_asym_first} shows the contour
plot of \eqref{gap:har_2D} and the lower bound in \eqref{lbdtl}
for different $\eta\ge1$ and $\alpha$. It shows that (i)
the asymptotic results in
\eqref{gap:har_2D}  degenerates to $0$
when either $\alpha\to0^+$   or $\eta\to +\infty$ (cf. Fig. \ref{fig:gap_har_2D_asym_first} (left)), and (ii)
the lower bound in \eqref{lbdtl}
does show the effect of the parameter $\eta\ge1$
properly since the contour line is almost
vertical when  $\eta\gg1$.

\subsection{Numerical results for general potentials}
Combining \eqref{lbdtl} and the scaling property in Lemma
  \ref{lem:scalew}, noting \eqref{eq:eigFsp} and \eqref{eq:eig} with \eqref{harm}, we can formally obtain a
  lower bound of the fundamental gap $\delta(\alpha)$ of the FSO \eqref{eq:eig} with \eqref{harm}
  \be\label{bdw876}
  \delta(\alpha)=\gm^{\frac{2\alpha}{2+\alpha}}\; \tdelta(\alpha)
  \ge 2^{\frac{4\alpha}{2+\alpha}}\,
\frac{ \alpha }{2+\alpha}\;
\frac{\gm^{\frac{2\alpha}{2+\alpha}}}{\eta^{\frac{2-\alpha}{2+\alpha}}}.
  \ee
To verify numerically the lower bound in \eqref{bdw876}, Fig. \ref{fig:gap_har_2D} shows numerical results of the fundamental gap
$\delta(\alpha)$ of \eqref{eq:eig} with \eqref{harm}.

\begin{figure}[htbp]
\centerline{\psfig{figure=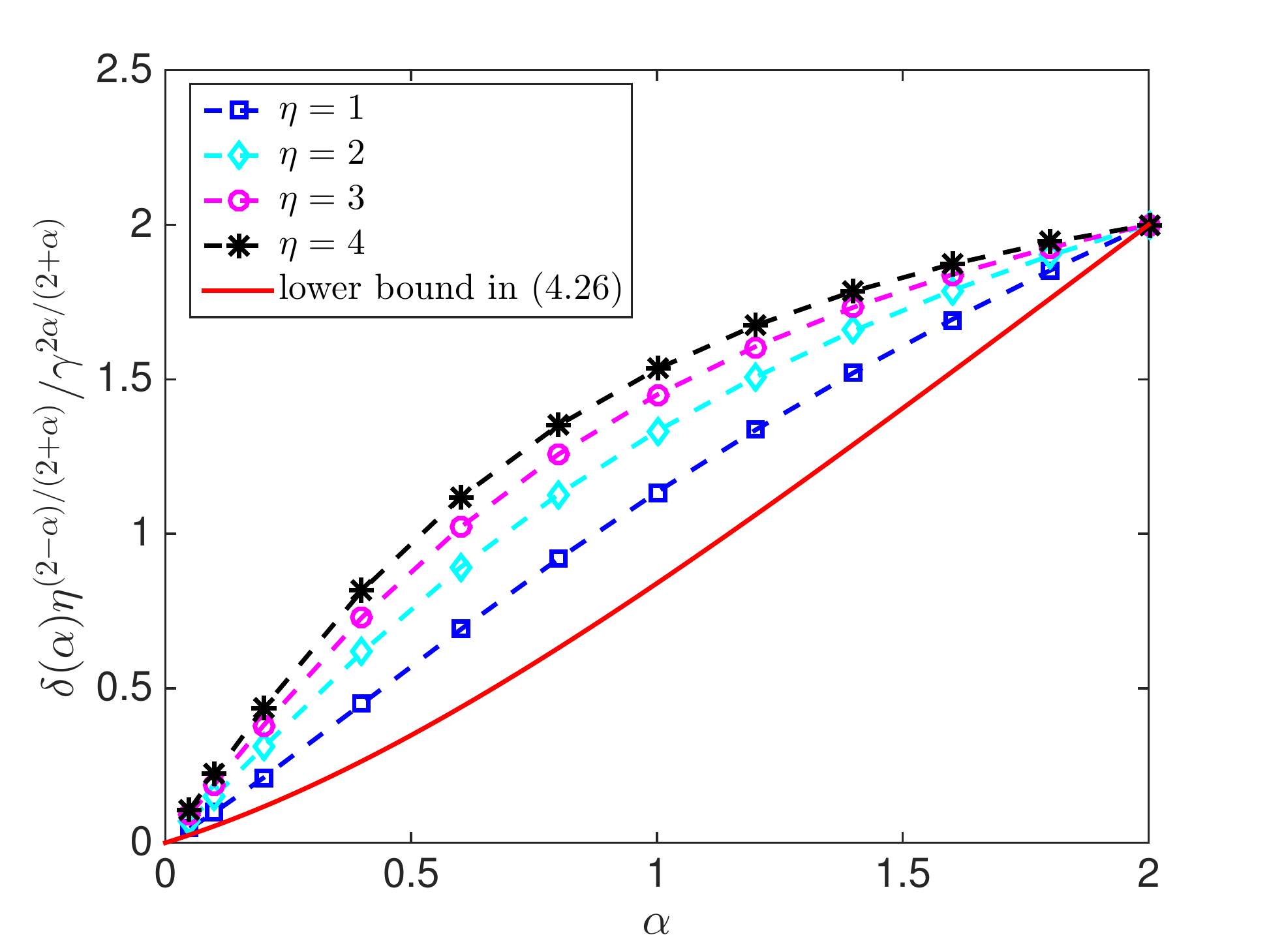,height=6cm,width=10cm,angle=0}}
\caption{Comparison of the lower bound in \eqref{bdw876} (solid line) and
numerical results (dash lines)  for the fundamental gap $\delta(\alpha)$ of the FSO \eqref{eq:eig} with $V(x,y)=\gm^2(x^2+\eta^2y^2)$ for different $\eta\ge1$ and $\gm>0$. }
\label{fig:gap_har_2D}
\end{figure}

Furthermore, to check numerically whether the lower bound
in \eqref{bdw876} is still valid for \eqref{eq:eig} with general convex harmonic-type potentials, Fig. \ref{fig:gap_har_2D_general}
shows numerical results of the fundamental gap
$\delta(\alpha)$ of \eqref{eq:eig} with different potentials taken as
Case I: $V(x,y)=2x^2 + 20y^2+\cos(x)+2\sin(2y)$ with $\gm=\sqrt{6}/2$ and $\eta=4$; and Case II: $V(x,y)=x^2+100y^2+\cos(x)+10\sin(2y)$ with
$\gm=\sqrt{2}/2$ and $\eta=4\sqrt{15}$.

\begin{figure}[htbp]
\centerline{\psfig{figure=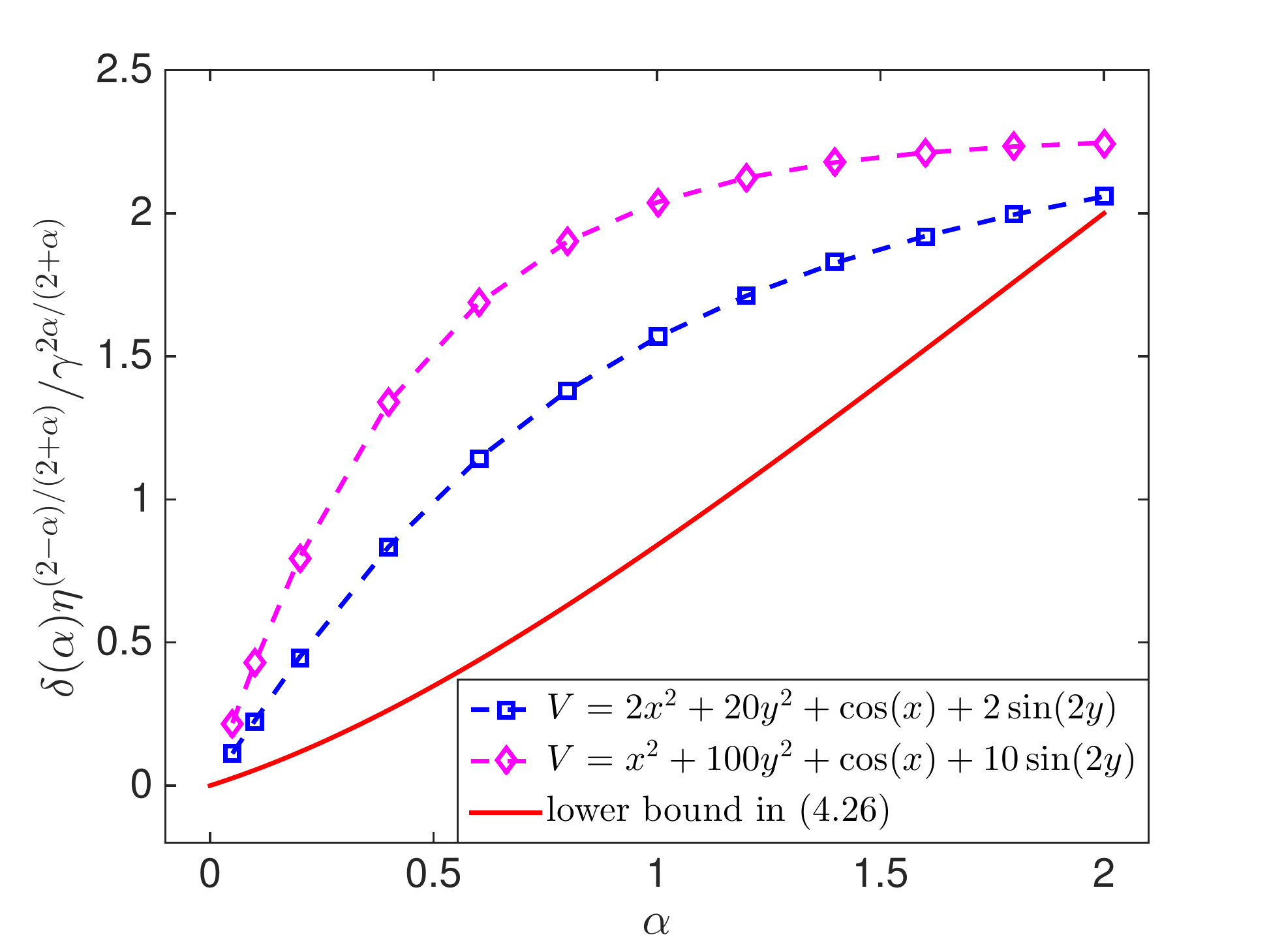,height=6cm,width=10cm,angle=0}}
\caption{Comparison of the lower bound in \eqref{bdw876} (solid line) and
numerical results (dash lines)  for the fundamental gap $\delta(\alpha)$ of the FSO \eqref{eq:eig} with different potentials $V(\bx)$. }
\label{fig:gap_har_2D_general}
\end{figure}

\bigskip

Based on the asymptotic results and numerical results in this section,
as well as extensive  numerical results which draw similar conclusion and thus are not shown here for
brevity, we can formulate the following:

\smallskip
\textbf{Gap Conjecture II} (For the FSO
\eqref{eq:eig} in the whole space). Assume
the potential $V(\mathbf{x})\in C^2({\mathbb R}^n)$
in \eqref{eq:eig} satisfies
\be\label{cond:V}
  \gm_1^2I_n\le \frac{1}{2}D^2V(\bx)\le \gm_2^2I_n, \qquad
  \bx\in{\mathbb R}^n,
  \ee
where $0<\gm_1\le \gm_2$ are two positive constants and $I_n$ is the
$n\times n$ identity matrix. Denote $\gm=\gm_1$ and set $\eta=\gm_2/\gm_1\ge1$, then the fundamental gap $\delta(\alpha)$ of
the FSO \eqref{eq:eig} can be bounded below by
 \be\label{gap_conj:har_2D}
\delta(\alpha)\ge2^{\frac{4\alpha}{2+\alpha}}\frac{\alpha}
{2+\alpha}\frac{\gm^{\frac{2\alpha}{2+\alpha}}}
{\eta^{\frac{2-\alpha}{2+\alpha}}}=2^{\frac{4\alpha}{2+\alpha}}\frac{\alpha}
{2+\alpha}\frac{\gm_1}{
\gm_2^{\frac{2-\alpha}{2+\alpha}}},\qquad 0<\alpha\le 2.
\ee

\subsection{Numerical results for well potential}
Consider a well potential in \eqref{eq:eig}
 \begin{gather}\label{def:well}
V(\bx)=
\begin{cases}
0,& \text{for } \bx\in\Omega,\\
V_0,& \text{for } \bx\in\Omega^c,
\end{cases}
\end{gather}
for some $V_0>0$. We solve \eqref{eq:eig} with \eqref{def:well} numerically
and compare the solutions with those in \eqref{eq:eigfso} and/or \eqref{eq:eigl} by letting
$V_0\to+\infty$.
Denote $0<E_1^{V_0}<E_2^{V_0}<\ldots$ be the eigenvalues of
\eqref{eq:eig} with \eqref{def:well} and $\phi_1^{V_0}(x)$, $\phi_2^{V_0}(x)$, $\ldots$
be the corresponding eigenfunctions.
Similarly, denote $0<\lmd_1<\lmd_2<\ldots$ be the eigenvalues of
\eqref{eq:eigfso} and $\phi_1(x)$, $\phi_2(x)$, $\ldots$
be the corresponding eigenfunctions; and denote
$0<\tilde \lmd_1<\tilde \lmd_2<\ldots$ be the eigenvalues of
\eqref{eq:eigl} and $\tilde \phi_1(x)$, $\tilde \phi_2(x)$, $\ldots$
be the corresponding eigenfunctions. All the solutions are obtained numerically.

 Fig. \ref{fig:numeric_box} shows $|E_1^{V_0}-\lmd_1|$ and
 $|E_1^{V_0}-\tilde \lmd_1|$ for different
 $0<\alpha\le 2$ and $V_0>0$. Similarly, Fig. \ref{fig:numeric_box1} shows $|\phi_1^{V_0}(x)-\phi_1(x)|$ and
 $|\phi_1^{V_0}(x)-\tilde \phi_1(x)|$ for $\alpha=1.5$ and different $V_0>0$.
Numerical comparisons were also performed for other eigenvalues and their
corresponding eigenfunctions, which draw similar conclusion and thus are not shown here for brevity.

\begin{figure}[htbp]
\centerline{\psfig{figure=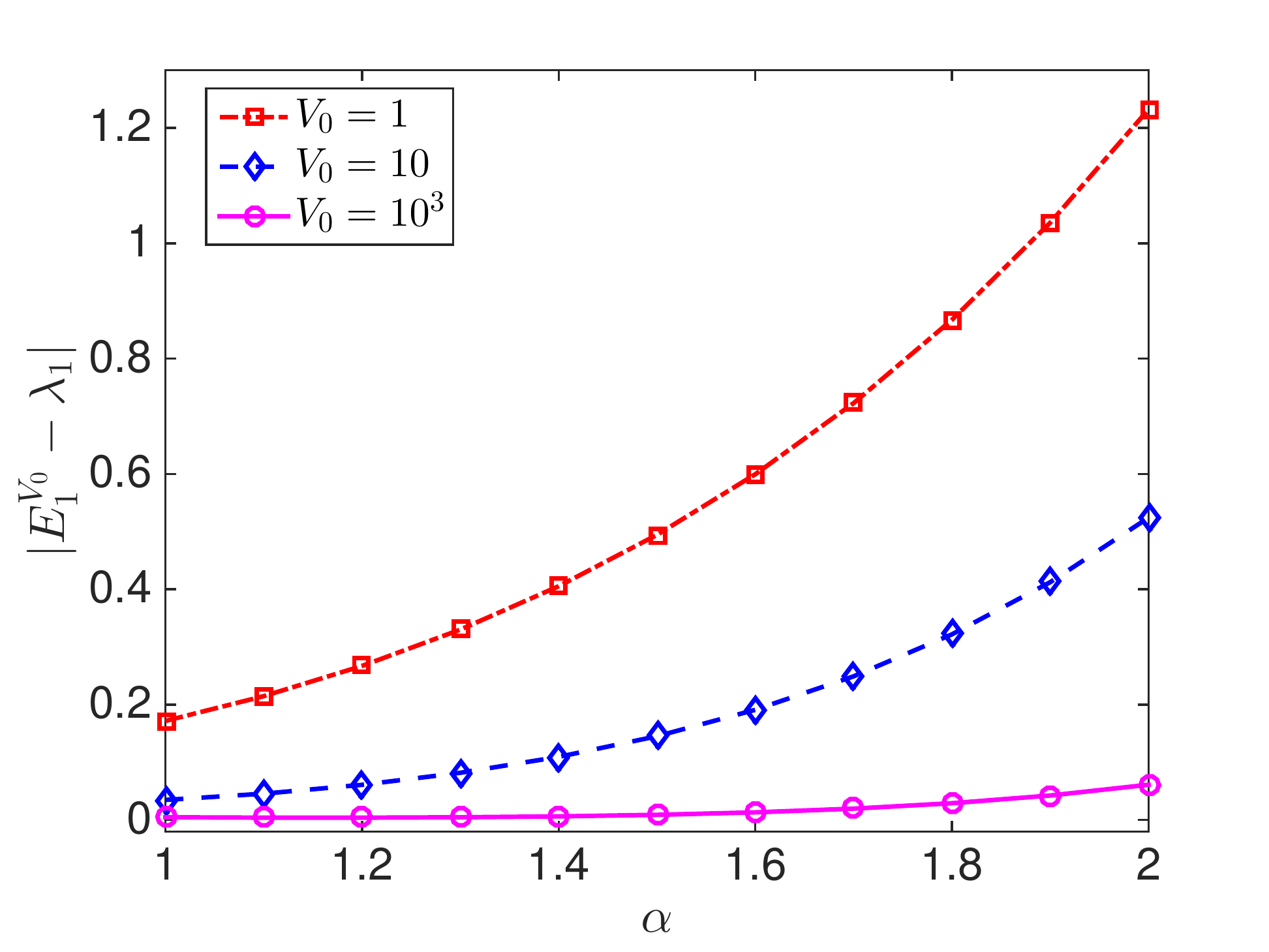,height=6cm,width=7cm,angle=0}
\psfig{figure=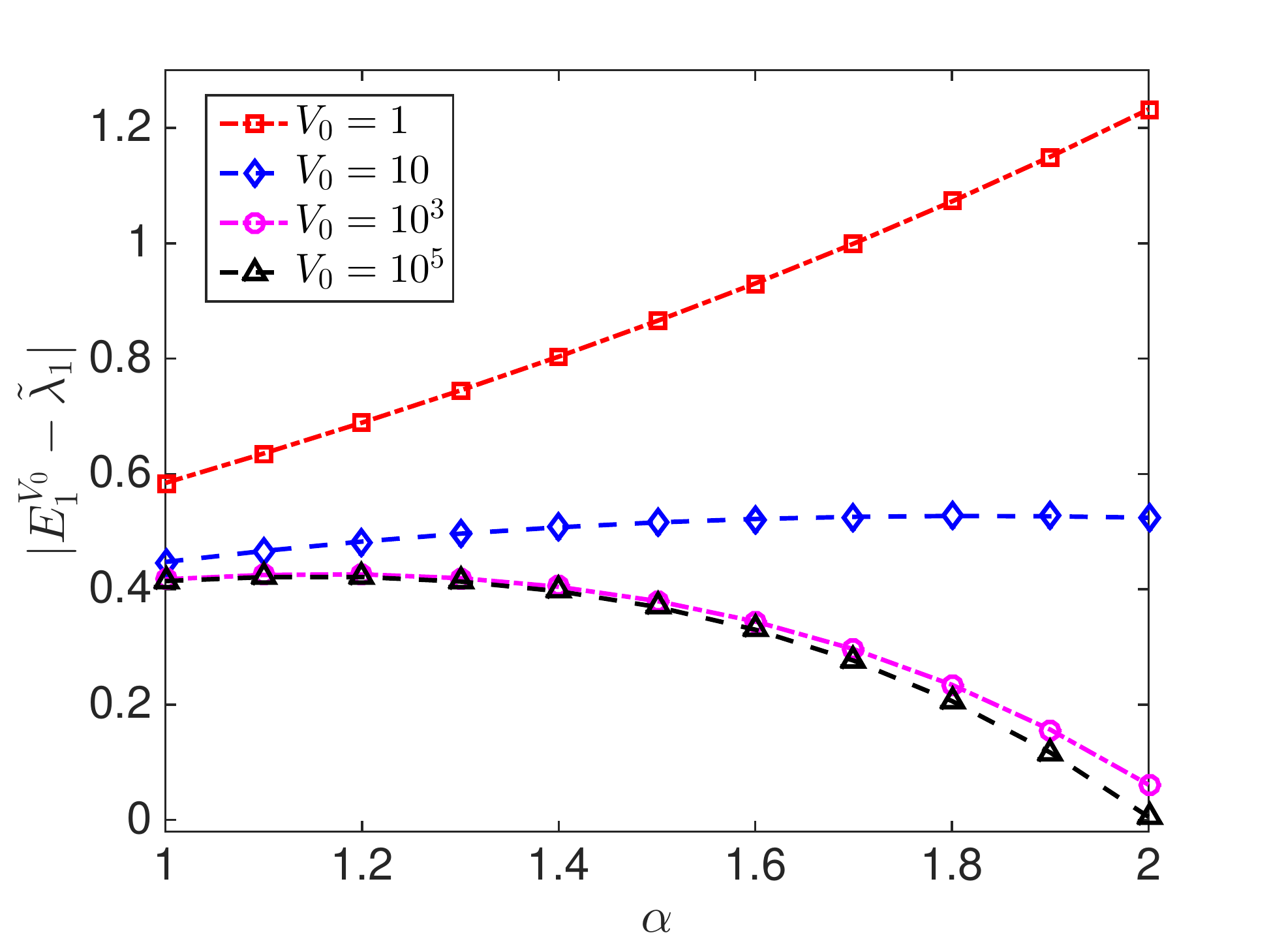,height=6cm,width=7cm,angle=0}}
\caption{Comparison of eigenvalues of
\eqref{eq:eig} with a box potential \eqref{def:well} and those of
\eqref{eq:eigfso} and/or \eqref{eq:eigl} for different $\alpha$ and $V_0$.}
\label{fig:numeric_box}
\end{figure}

\begin{figure}[htbp]
\centerline{\psfig{figure=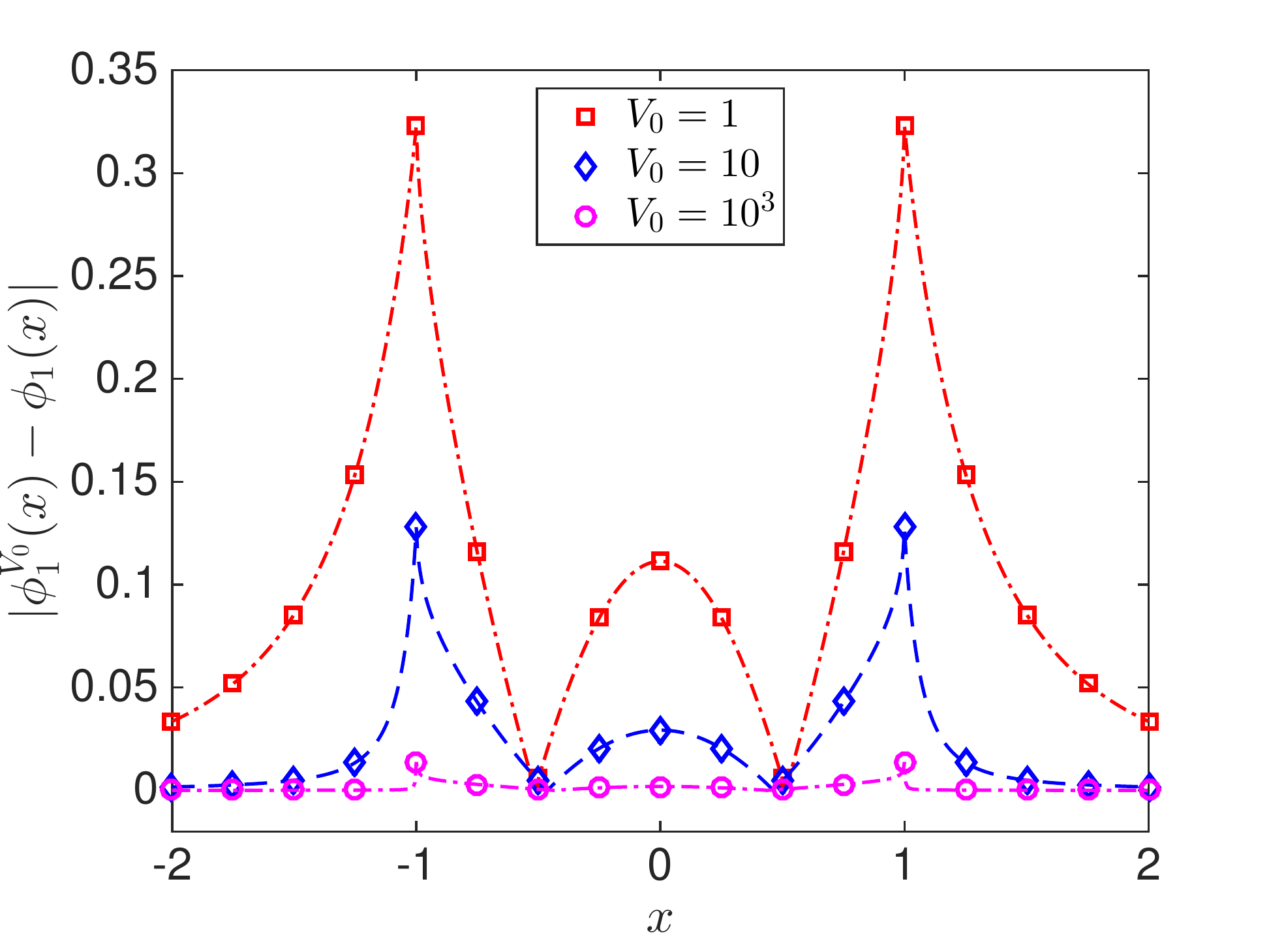,height=6cm,width=7cm,angle=0}
\psfig{figure=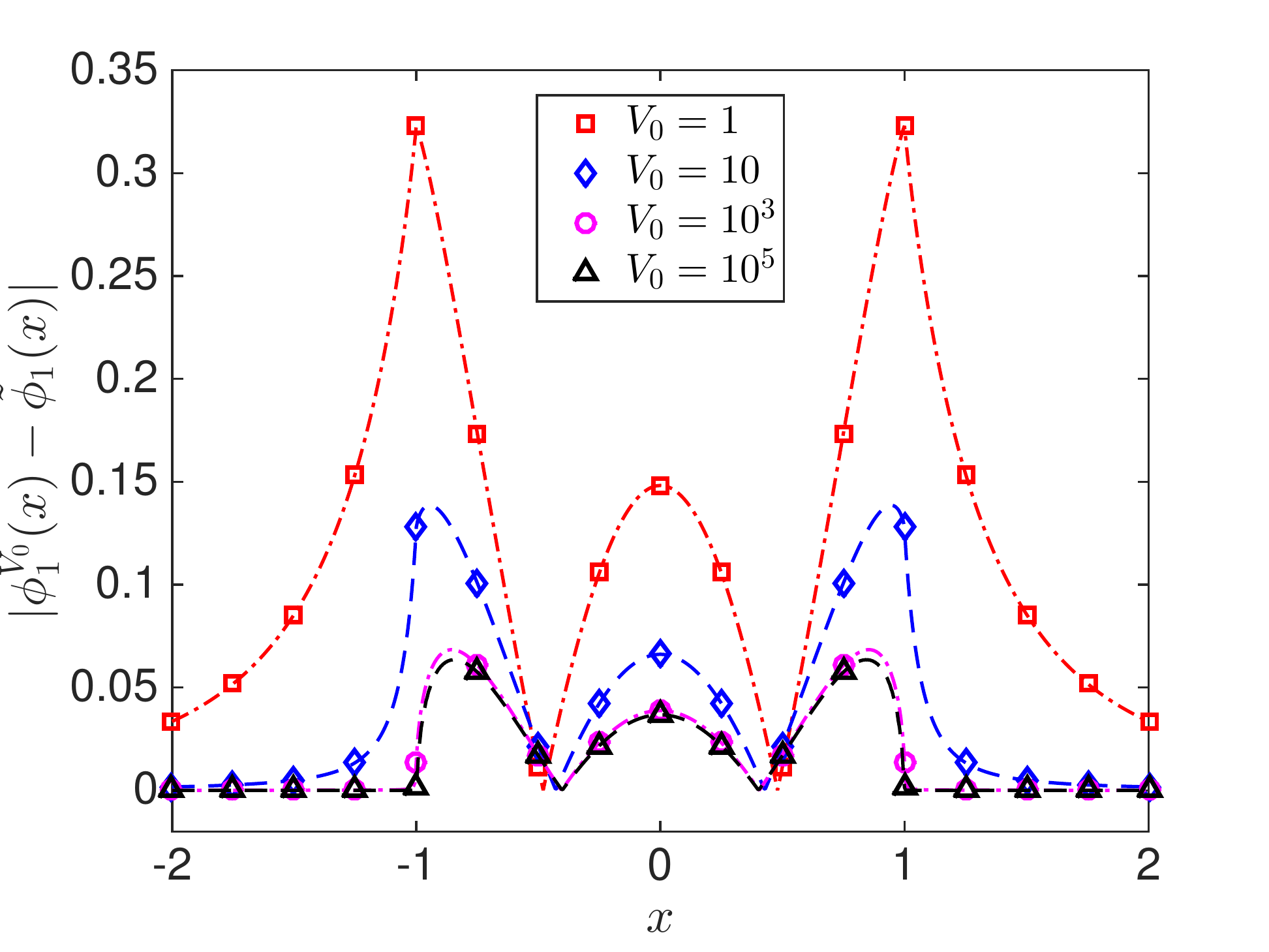,height=6cm,width=7cm,angle=0}}
\caption{Comparison of eigenfunctions of
\eqref{eq:eig} with a box potential \eqref{def:well} and those of
\eqref{eq:eigfso} and/or \eqref{eq:eigl} for $\alpha=1.5$ and different
$V_0$.}
\label{fig:numeric_box1}
\end{figure}

From Figs. \ref{fig:numeric_box}\&\ref{fig:numeric_box1} and additional numerical results which draw similar conclusion and thus are not shown here for brevity,
when $\alpha=2$, the eigenvalues and their corresponding eigenfunctions
of  \eqref{eq:eig} with \eqref{def:well} converge to those
of \eqref{eq:eigfso} and \eqref{eq:eigl} when $V_0\to+\infty$.
However, when $0<\alpha<2$, the eigenvalues and their corresponding eigenfunctions
of  \eqref{eq:eig} with \eqref{def:well} converge to those
of \eqref{eq:eigfso} when $V_0\to+\infty$, and they don't converge
to those of \eqref{eq:eigl}!

\section{The fundamental gaps of the FSO \eqref{eq:fso} on bounded domains with periodic boundary conditions}\label{sec:periodic}
Take $\Omega=\prod_{j=1}^n(0,L_j)$ and $V(\bx)$ be a periodic function
with respect to $\Omega$ in \eqref{eq:eig}.
Without loss of generality, we assume $L_1\ge L_2 \ge \ldots \ge L_n>0$
and $V_\bog(\bx):=V(\bx)|_\Omega \ge0$.
In this case, \eqref{eq:eig} can be reduced to
\begin{equation}\label{eq:eigper}
\begin{split}
&L_{\rm Per}\,\phi(\bx):=\left[(-\Delta)^{\frac{\alpha}{2}}
+V_\bog(\bx)\right]\phi(\bx)=
\lmd\,\phi(\bx),\qquad \bx\in \Omega,\\
&\left.\phi(\bx)\right|_{\partial \Omega} \ \hbox{is periodic}.
\end{split}
\end{equation}
In this case, the two definitions of the fractional Laplacian operator \eqref{def:FL_I} and \eqref{lfl11} are equivalent for $0<\alpha\le 2$ \cite{Roncal,Roncal2}. Let $0<\lmd_1:=\lmd_1(\alpha)<\lmd_2:=\lmd_2(\alpha)$ be the first two smallest positive eigenvalues
of \eqref{eq:eigper}, then the fundamental gap of \eqref{eq:eigper}
is denoted as:
\be\label{gap1dpd}
\delta_{\rm per}(\alpha):=\lmd_2(\alpha)-\lmd_1(\alpha), \qquad 0<\alpha\le 2.
\ee

  Similar to proof of Lemmas \ref{lem:scale_loc}\&\ref{lem:scalefso},
we can obtain
the following scaling property (the proof is omitted here for brevity).

\begin{lemma}\label{lem:scaleper}
Let $\lambda$ be an eigenvalue of \eqref{eq:eigper} and $\phi:=\phi(\bx)$ is the corresponding eigenfunction, under the transformation \eqref{scaxp}, then $\tilde \lambda =D^\alpha \lambda$
 and  $\tilde \phi:=\tilde \phi(\tbx) = \phi(D\tbx)$
are the eigenvalue and the  corresponding eigenfunction
of the following eigenvalue problem
\begin{equation}\label{eq:eigpernew}
\begin{split}
&\tilde L_{\rm Per}\,\tilde\phi(\tbx):=\left[(-\Delta)^{\frac{\alpha}{2}}
+\tilde V_\botg(\tbx)\right]\tilde\phi(\tbx)=
\tilde \lmd\,\tilde \phi(\tbx),\qquad \tbx\in \tilde \Omega,\\
&\left.\tilde \phi(\tbx)\right|_{\partial \tilde \Omega} \ \hbox{is periodic}
\end{split}
\end{equation}
which immediately imply the scaling property on the fundamental gap $\delta_{\rm per}(\alpha)$  of \eqref{eq:eigper} as
\be
\delta_{\rm per}(\alpha) = \frac{\tdelta_{\rm per}(\alpha)}{D^{\alpha}}, \qquad
0<\alpha\le 2,
\ee
where $\tdelta_{\rm per}(\alpha)$ is the fundamental gap of
\eqref{eq:eigpernew} with the diameter of $\tilde \Omega$ as $1$.
\end{lemma}

\smallskip

\begin{lemma}\label{per1d}
Take $n=1$ and $V_\bog(x)\equiv 0$ in \eqref{eq:eigper}, then we have
\be\label{gap1dp789}
\delta_{\rm per}(\alpha)=\frac{(2\pi)^{\alpha}(2^{\alpha}-1)}{L_1^{\alpha}},
\qquad 0<\alpha\le 2.
\ee
\end{lemma}

\smallskip

\begin{proof} When $n=1$ and $V_\bog(x)\equiv 0$ in \eqref{eq:eigper}, we know that the first three eigenvalues and their corresponding eigenfunctions
can be taken as \cite{BC,BL}
\be\label{per1d567}
\begin{split}
&E_0:=E_0(\alpha)=0, \qquad \phi_0^{(\alpha)}(x)\equiv\frac{1}{\sqrt{L_1}},\\
&E_1:=E_1(\alpha)=\left(\frac{2\pi}{L_1}\right)^{\alpha}, \quad
\phi_1^{(\alpha)}(x)=\sqrt{\frac{2}{L_1}}\sin\left(\frac{2\pi x}{L_1}\right),\qquad 0\le x\le L_1,\\
&E_2:=E_2(\alpha)=\left(\frac{4\pi}{L_1}\right)^{\alpha},\quad
\phi_2^{(\alpha)}(x)=\sqrt{\frac{2}{L_1}}\sin\left(\frac{4\pi x}{L_1}\right).
\end{split}
\ee
Plugging \eqref{per1d567} into \eqref{gap1dpd}, we obtain \eqref{gap1dp789}
immediately.
\end{proof}

\smallskip

\begin{lemma}\label{per2d}
Take $n=2$ and $V_\bog(\bx)\equiv 0$ in \eqref{eq:eigper}, then we have
\be\label{gap2dp789}
\delta_{\rm per}(\alpha)=\begin{cases}
\frac{(2\pi)^{\alpha}(2^{\alpha/2}-1)}{L_1^{\alpha}}, &\text{ if } L_1=L_2, \\
\frac{(2\pi)^{\alpha}}{L_2^{\alpha}}-\frac{(2\pi)^{\alpha}}{L_1^{\alpha}}, &\text{ if } L_2<L_1\le2L_2,\\
\frac{(2\pi)^{\alpha}(2^{\alpha}-1)}{L_1^{\alpha}}, &\text{ if } L_1\ge2L_2.
\end{cases}
\ee
\end{lemma}

\smallskip
\begin{proof}
When $n=2$ and $V_\bog(\bx)\equiv 0$ in \eqref{eq:eigper}, when $L_1=L_2$,
we know that the first three eigenvalues and their corresponding eigenfunctions
can be taken as \cite{BC,BL}
\be\label{per2d567}
\begin{split}
&E_0:=E_0(\alpha)=0, \qquad \phi_0^{(\alpha)}(\bx)\equiv A_0:=\frac{1}{\sqrt{\prod_{j=1}^2L_j}},\\
&E_1:=E_1(\alpha)=\left(\frac{2\pi}{L_1}\right)^{\alpha}, \quad
\phi_1^{(\alpha)}(\bx)=\sqrt{2}A_0\sin\left(\frac{2\pi x}{L_1}\right),\qquad \bx=(x,y)^T\in \Omega,\\
&E_2:=E_2(\alpha)=E_2(\alpha)=\left(\frac{2\sqrt{2}\pi}
{L_1}\right)^{\alpha},\quad
\phi_2^{(\alpha)}(\bx)=2A_0\sin\left(\frac{2\pi x}{L_1}\right)\sin\left(\frac{2\pi y}{L_2}\right).
\end{split}
\ee
Plugging \eqref{per2d567} into \eqref{gap1dpd}, we obtain \eqref{gap2dp789}
when $L_1=L_2$
immediately. Similarly, when $L_1>L_2$, we get
\be\label{per2d5678}
\begin{split}
&E_0(\alpha)=0, \qquad \phi_0^{(\alpha)}(\bx)\equiv A_0:=\frac{1}{\sqrt{\prod_{j=1}^2L_j}},\\
&E_1(\alpha)=\left(\frac{2\pi}{L_1}\right)^{\alpha}, \quad
\phi_1^{(\alpha)}(\bx)=\sqrt{2}A_0\sin\left(\frac{2\pi x}{L_1}\right),\qquad \bx=(x,y)^T\in \Omega,\\
&E_2(\alpha)=\left\{
\begin{array}{l}
\left(\frac{4\pi}{L_1}\right)^{\alpha},\\
\left(\frac{2\pi}{L_2}\right)^{\alpha},\\
\end{array}\right.
\
\phi_2^{(\alpha)}(\bx)=\left\{\begin{array}{ll}
\sqrt{2}A_0\sin\left(\frac{4\pi x}{L_1}\right),  &\text{ if } L_1\ge2L_2,\\
\sqrt{2}A_0\sin\left(\frac{2\pi y}{L_2}\right), &\text{ if } L_2<L_1\le2L_2.\\
\end{array}\right.
\end{split}
\ee
Plugging \eqref{per2d5678} into \eqref{gap1dpd}, we obtain \eqref{gap2dp789}
when $L_1=L_2$
immediately.
\end{proof}

For the convenience of readers, Fig. \ref{fig:periodic_2D}
shows the phase diagram of the first several
eigenvalues and their corresponding eigenfunctions
of \eqref{eq:eigper}  with respect to $L_1/L_2$ when $n=2$.

\begin{figure}[htbp]
\centerline{\psfig{figure=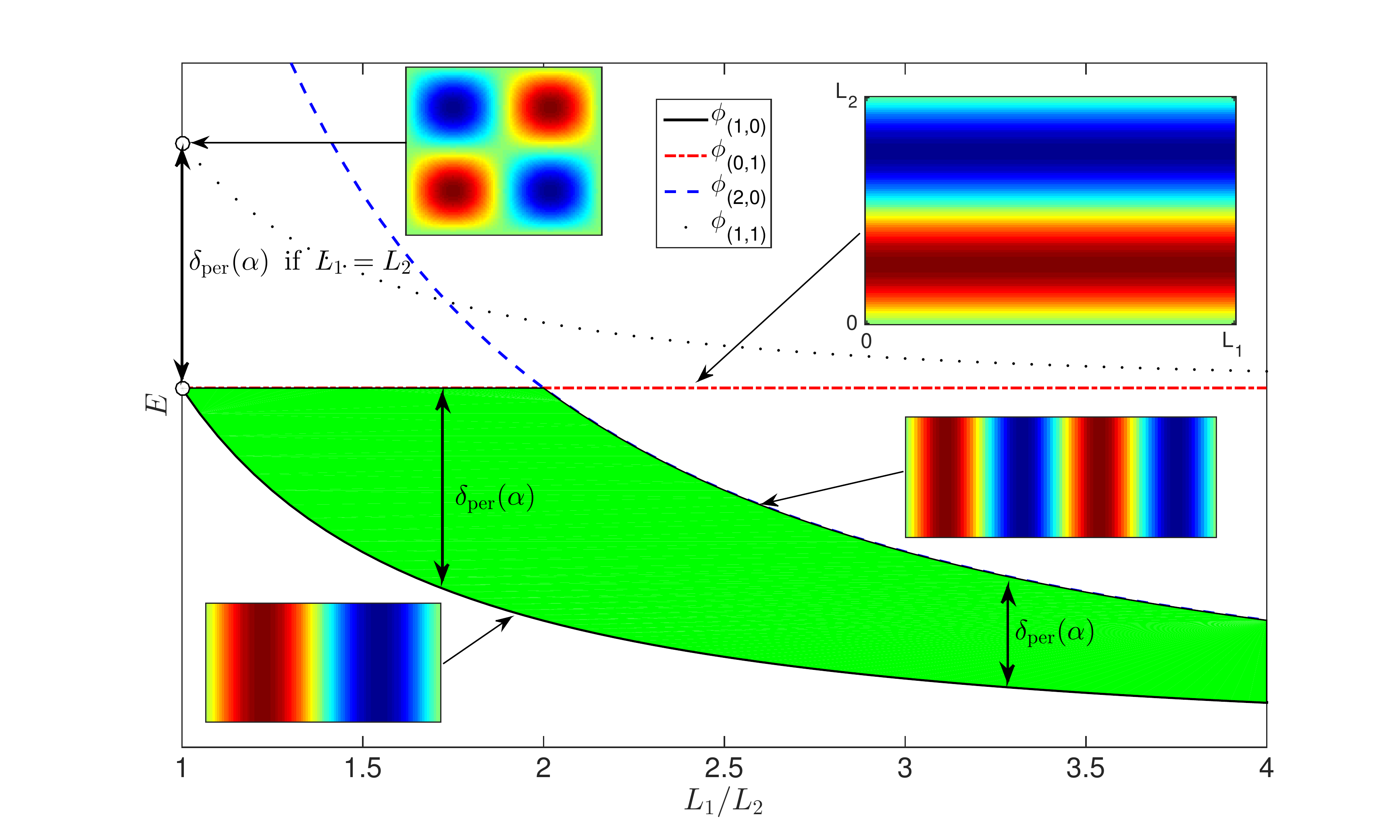,height=8cm,width=14cm,angle=0}}
\caption{Phase diagram of the first several eigenvalues
and their corresponding eigenfunctions of \eqref{eq:eigper}
with $n=2$, $\alpha=1.5$ and $V_\bog(\bx)\equiv 0$
for different $L_1/L_2$.
Obviously, for different ratios $L_1/L_2$, the choice of the second excited state $\phi_2^{(\alpha)}$ is different. The green part denotes the fundamental gap
$\delta_{\rm per}(\alpha)$ for $L_1>l_2$. }
\label{fig:periodic_2D}
\end{figure}

\section{Conclusion}\label{sec:conclusion}
By using asymptotic and numerical methods,
we obtain the fundamental gaps of the fractional
Schr\"{o}dinger operator (FSO) in different cases including
the local FSO on bounded domains, the FSO on bounded domains with zero
extension outside the domains, the FSO in the whole space,
and the FSO on bounded domains with periodic boundary conditions.
Based on our asymptotic and numerical results,
we formulate gap conjectures of the fundamental gap of the
FSO in different cases. The gap conjectures link the algebraic
property -- difference of the first two smallest eigenvalues of
the eigenvalue problem -- and the geometric property -- diameters
of the bounded domains.

\end{document}